\documentclass[11pt]{article}

\usepackage{amsfonts,amssymb,latexsym,a4wide,exscale,enumerate}
\usepackage{amsmath}
\usepackage{amsthm}

\usepackage{graphicx}
\usepackage[all]{xy}
\SelectTips{cm}{}

\newcommand{\seq}{{\rm Seq}}

\newcommand{\ii}{ \textbf{\textit{i}}}
\newcommand{\jj}{ \textbf{\textit{j}}}
\newcommand{\kk}{ \textbf{\textit{k}}}
\newcommand{\conf}{{\rm conf}}
\newcommand{\con}{\xy (-4,0)*{}; (4.5,0)*{}; (-3,0)*{};(3,0)*{} **\dir{-};  \endxy}

\newcommand{\refequal}[1]{\xy {\ar@{=}^{#1}
(-1,0)*{};(1,0)*{}};
\endxy}

\usepackage{fancyheadings}
\newcommand{\BOX}{\hbox {$\sqcap$ \kern -1em $\sqcup$}}

\renewcommand{\to}{\rightarrow}
\newcommand{\maps}{\colon}

\newcommand{\scs}{\scriptstyle}

\theoremstyle{definition}
\newtheorem{thm}{Theorem}[section]
\newtheorem{cor}[thm]{Corollary}

\newtheorem{lem}[thm]{Lemma}
\newtheorem{rem}[thm]{Remark}
\newtheorem{prop}[thm]{Proposition}
\newtheorem{defn}[thm]{Definition}
\newtheorem{example}[thm]{Example}

\numberwithin{equation}{section}

\let\hat=\widehat

\let\phi=\varphi
\let\epsilon=\varepsilon

\usepackage{bbm}

\def\N{{\mathbbm N}}
\def\R{{\mathbbm R}}
\def\Z{{\mathbbm Z}}

\def\cal#1{\mathcal{#1}}%
\def\1{\mathbbm{1}}%
\def\nn{\notag}

\def\mf{\mathfrak}

\newcommand{\up}[1]{\xybox{
   (-3,-13)*{};
  (3,8)*{};
 (0,0)*{\includegraphics[scale=0.5]{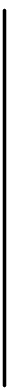}};
 (-.1,-12)*{\scs #1};
 }}

 \newcommand{\updot}[1]{\xybox{
   (-3,-13)*{};
  (3,8)*{};
 (0,0)*{\includegraphics[scale=0.5]{up.eps}};
 (-.1,-12)*{\scs #1};(0,0)*{\bullet};
 }}

\renewcommand{\sup}[1]{\xybox{
   (-3,-7)*{};
  (3,6)*{};
 (0,0)*{\includegraphics[scale=0.5]{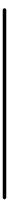}};
 (-.1,-7)*{\scs #1};
 }}

 \newcommand{\supdot}[1]{\xybox{
   (-3,-7)*{};
  (3,6)*{};
 (0,0)*{\includegraphics[scale=0.5]{short_up.eps}};
 (-.1,-7)*{\scs #1}; (0,0)*{\bullet};
 }}

\newcommand{\dcross}[2]{\xybox{
 (-6,-7)*{};
 (6,6)*{};
 (0,0)*{\includegraphics[scale=0.5]{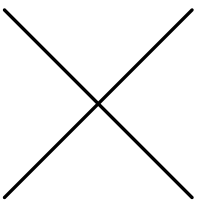}};
 (-5.1,-7)*{\scs #1};
 (4.7,-7)*{\scs #2};
}}

\newcommand{\dcrossul}[2]{\xybox{
(-2.5,2.5)*{\bullet};
 (-6,-7)*{};
 (6,6)*{};
 (0,0)*{\includegraphics[scale=0.5]{dcross.eps}};
 (-5.1,-7)*{\scs #1};
 (4.7,-7)*{\scs #2};
}}

\newcommand{\dcrossur}[2]{\xybox{
(2.5,2.5)*{\bullet};
 (-6,-7)*{};
 (6,6)*{};
 (0,0)*{\includegraphics[scale=0.5]{dcross.eps}};
 (-5.1,-7)*{\scs #1};
 (4.7,-7)*{\scs #2};
}}

\newcommand{\dcrossdl}[2]{\xybox{
(-2.5,-2.5)*{\bullet};
 (-6,-7)*{};
 (6,6)*{};
 (0,0)*{\includegraphics[scale=0.5]{dcross.eps}};
 (-5.1,-7)*{\scs #1};
 (4.7,-7)*{\scs #2};
}}

\newcommand{\dcrossdr}[2]{\xybox{
(2.5,-2.5)*{\bullet};
 (-6,-7)*{};
 (6,6)*{};
 (0,0)*{\includegraphics[scale=0.5]{dcross.eps}};
 (-5.1,-7)*{\scs #1};
 (4.7,-7)*{\scs #2};
}}

\newcommand{\twocross}[2]{\xybox{
 (-6,-13)*{};
 (6,8)*{};
 (0,0)*{\includegraphics[scale=0.5]{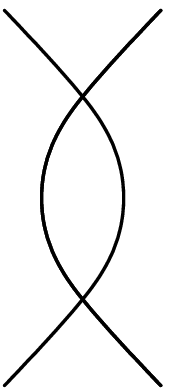}};
 (-4.1,-12)*{\scs #1};
 (3.7,-12)*{\scs #2};
}}

\newcommand{\linecrossL}[3]{\xybox{
 (-6,-13)*{};
 (6,8)*{};
 (0,0)*{\includegraphics[scale=0.5]{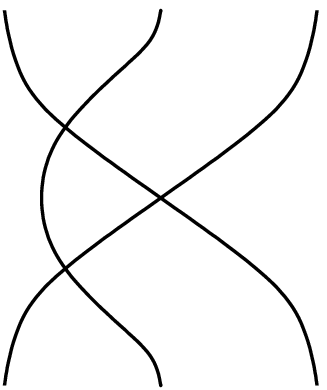}};
 (-8,-12)*{\scs #1};
 (0,-12)*{\scs #2};
 (8,-12)*{\scs #3};
}}

\newcommand{\linecrossR}[3]{\xybox{
 (-6,-13)*{};
 (6,8)*{};
 (0,0)*{\includegraphics[angle=180, scale=0.5]{line_crossL.eps}};
 (-8,-12)*{\scs #1};
 (0,-12)*{\scs #2};
 (8,-12)*{\scs #3};
}}

\title{Nilpotency in type $A$ cyclotomic quotients}
\author{Alexander E. Hoffnung and Aaron D. Lauda}

\date{\today}

\begin{document}
%

\maketitle

\begin{abstract}
We prove a conjecture of Brundan and Kleshchev on the nilpotency degree of cyclotomic quotients of rings that categorify one-half of quantum $sl(k)$.
\end{abstract}
\newcommand{\Ra}{R'}
%
%

%
\section{Introduction}
%

Let $\Gamma$ denote the quiver associated to a simply-laced Kac-Moody algebra
$\mf{g}$.  Let $\Z[I]$ denote the free abelian group on the set of vertices $I$ of $\Gamma$.
There is a bilinear Cartan form on $\Z[I]$ given on the basis elements $i,j\in I$ by
$$
      i\cdot j = \begin{cases} 2 & \textrm{if $i=j$}, \\
      -1 & \textrm{if $i$ and $j$ are joined by an edge}, \\
      0 & \textrm{otherwise}. \end{cases}
$$
We sometimes write $ i \con j$ for $i \cdot j =-1$.

For a Kac-Moody Lie algebra $\mf{g}$ associated to an arbitrary Cartan datum,  a graded algebra $R$ was defined in \cite{KL,KL2} and shown to categorify $U^-_q(\mf{g})$, the integral form of the negative half of the quantum universal enveloping algebra.  These algebras also appear in a categorification of the entire quantum group~\cite{KL3}, and in the 2-representation theory of Kac-Moody algebras~\cite{Ro}. Given a field $\Bbbk$, the $\Bbbk$-algebra $R$ is defined by finite $\Bbbk$-linear combinations of braid--like diagrams in the plane, where each strand is coloured by a vertex $i \in I$.  Strands can intersect and can carry dots; however, triple intersections are not allowed.  Diagrams are considered up to planar isotopy that do not change the combinatorial type of the diagram.  We recall the local relations for simply-laced Cartan datum:
\begin{eqnarray} \label{eq_UUzero}
   \xy   (0,0)*{\twocross{i}{j}}; \endxy
 & = & \left\{
\begin{array}{ccc}
  0 & \qquad & \text{if $i=j$, } \\ \\
  \xy (0,0)*{\sup{i}};  (8,0)*{\sup{j}};  \endxy
  & &
 \text{if $i \cdot j=0$, }
  \\    \\
  \xy  (0,0)*{\supdot{i}};   (8,0)*{\sup{j}};  \endxy
  \quad + \quad
   \xy  (8,0)*{\supdot{j}};  (0,0)*{\sup{i}};  \endxy
 & &
 \text{if $i \cdot j=-1$. }
\end{array}
\right.
\end{eqnarray}
\begin{eqnarray}\label{eq_ijslide}
  \xy  (0,0)*{\dcrossul{i}{j}};  \endxy
 \quad  = \;\;
   \xy  (0,0)*{\dcrossdr{i}{j}};   \endxy
& \quad &
   \xy  (0,0)*{\dcrossur{i}{j}};  \endxy
 \quad = \;\;
   \xy  (0,0)*{\dcrossdl{i}{j}};  \endxy \qquad \text{for $i \neq j$}
\end{eqnarray}
\begin{eqnarray}        \label{eq_iislide1}
 \xy  (0,0)*{\dcrossul{i}{i}}; \endxy
    \quad - \quad
 \xy (0,0)*{\dcrossdr{i}{i}}; \endxy
  & = &
 \xy (-3,0)*{\sup{i}}; (3,0)*{\sup{i}}; \endxy \\      \label{eq_iislide2}
  \xy (0,0)*{\dcrossdl{i}{i}}; \endxy
 \quad - \quad
 \xy (0,0)*{\dcrossur{i}{i}}; \endxy
  & = &
 \xy (-3,0)*{\sup{i}}; (3,0)*{\sup{i}}; \endxy
\end{eqnarray}
\begin{eqnarray}      \label{eq_r3_easy}
\xy  (0,0)*{\linecrossL{i}{j}{k}}; \endxy
  &=&
\xy (0,0)*{\linecrossR{i}{j}{k}}; \endxy
 \qquad \text{unless $i=k$ and $i \cdot j=-1$   \hspace{1in} }
\\                   \label{eq_r3_hard}
\xy (0,0)*{\linecrossL{i}{j}{i}}; \endxy
  &-&
\xy (0,0)*{\linecrossR{i}{j}{i}}; \endxy
 \quad = \quad
 \xy  (-9,0)*{\up{i}};  (0,0)*{\up{j}}; (9,0)*{\up{i}}; \endxy
 \qquad \text{if $i \cdot j=-1$ }
\end{eqnarray}
Multiplication is given by concatenation of diagrams.  For more details see \cite{KL,KL2}.  The results in this note do not depend on the ground field $\Bbbk$; they remain valid when considering the ring $R$ defined as above with $\Z$-linear combinations of diagrams.

For $\nu=\sum_{i \in I} \nu_i \cdot i \in \N[I]$ write $\seq(\nu)$ for the subset of $I^m$ consisting of those sequences of vertices $\ii = i_1i_2\cdots i_m$ where $i_k\in I$ and vertex $i$ appears $\nu_i$ times.  The length $m$ of the sequence is equal to $|\nu|$.  Define
${\rm Supp}(\nu):=\{ i \mid \nu_i \neq 0\}$. The ring $R$ decomposes as
\begin{equation}
R = \bigoplus_{\nu \in \N[I]} R(\nu)
\end{equation}
where $R(\nu)$ is the subring generated by diagrams that contain $\nu_i$ strands coloured $i$ for each $i\in{\rm Supp}(\nu) $.  We write $1_{\ii}$ for the diagram with only vertical lines and no crossings, where the strands are coloured by the sequence $\ii$. The element $1_{\ii}$ is an idempotent of the
ring $R(\nu)$.  The rings $R(\nu)$ decompose further as
\begin{equation}
R(\nu) = \bigoplus_{\ii,\jj \in
\seq(\nu) }{_{\jj}R(\nu)_{\ii}}
\end{equation}
where ${_{\jj}R(\nu)_{\ii}} := 1_{\jj}R(\nu)1_{\ii}$ is the abelian group of all
linear combinations of diagrams with sequence $\ii$ at the bottom and sequence
$\jj$ at the top modulo the above relations.

Sometimes it is convenient to convert from graphical to algebraic notation. For a
sequence $\ii=i_1 i_2\dots i_m\in \seq(\nu)$ and $1\le r \le m$ we denote
\begin{eqnarray} \label{eq_dot_xki}
  x_{r,\ii} \quad := \quad
  \xy (0,0)*{\sup{}};  (0,-6)*{\scs i_1};  \endxy
    \dots
    \xy (0,0)*{\supdot{}};  (0,-6)*{\scs i_r}; \endxy
   \dots
    \xy (0,0)*{\sup{}};  (1,-6)*{\scs i_m}; \endxy
\end{eqnarray}
and
\begin{eqnarray} \label{eq_crossing_delta}
  \delta_{r,\ii} \quad := \quad
  \xy  (0,0)*{\sup{}};  (0,-6)*{\scs i_1};   \endxy
    \dots
  \xy (0,0)*{\dcross{i_r}{\; \; i_{r+1}}}; \endxy
   \dots
  \xy (0,0)*{\sup{}}; (1,-6)*{\scs i_m}; \endxy
\end{eqnarray}
The symmetric group $S_m$, where $m=|\nu|$, acts on $\seq(\nu)$  by permutations.  The transposition $s_r=(r,r+1)$ switches entries $i_r, i_{r+1}$ of $\ii$.  Thus, $\delta_{r,\ii} \in {_{s_r(\ii)}R(\nu)_{\ii}}$.

For $\Lambda=\sum_{i \in I} \lambda_i \cdot i \in \N[I]$ the ${\rm level}$ of $\Lambda$ is $\ell(\Lambda) = \sum_{i \in I} \lambda_i$. Let $\cal{J}_{\Lambda}$ be the ideal of $R(\nu)$ generated by elements $x_{1,\ii}^{\lambda_{i_1}}$ over all sequences $\ii=i_1\dots i_m \in \seq(\nu)$. Define the cyclotomic quotient of
the ring $R(\nu)$ at weight $\Lambda$ as the quotient
\begin{equation}
R_{\nu}^{\Lambda} := R_{\nu}/ \cal{J}_{\Lambda}.
\end{equation}
In terms of the graphical calculus the cyclotomic quotient $R_{\nu}^{\Lambda}$ is
the quotient of $R(\nu)$ by the ideal generated by
\begin{eqnarray}
  \xy (0,0)*{\supdot{i_1}}; (-4,2)*{\lambda_{i_1}};
  (6,0)*{\sup{\;\; i_2}};
  (12,0)*{\cdots};
  (18,0)*{\sup{\;\;\; i_m}}
  \endxy &=& 0 \label{eq_quotient1}
\end{eqnarray}
over all sequences $\ii$ in $\seq(\nu)$.  It was conjectured in \cite{KL} that $R_{\nu}^{\Lambda}$ categorifies the integrable representations of $U_q(\mf{g})$ of highest weight $\Lambda$. The
quotients $R_{\nu}^{\Lambda}$ are called cyclotomic quotients because they should
be the analogues of the Ariki-Koike cyclotomic Hecke algebras for other types.

This conjecture has been proven by Kleshchev and Brundan in type $A$~\cite{BK1, BK2}.  They construct an isomorphism
\[
 \xy {\ar^{\sim}(-8,0)*+{ R_{\nu}^{\Lambda}}; (8,0)*+{H_{\nu}^{\Lambda}},}; \endxy
\]
where $H_{\nu}^{\Lambda}$ is a block of the cyclotomic affine Hecke algebra $H_{m}^{\Lambda}$.   Ariki's categorification theorem~\cite{Ari1} gives an isomorphism between the integrable highest weight representation $V(\Lambda)$ for $U(\hat{\mathfrak{sl}_e})$ and the Grothendieck ring $\bigoplus_{m} K_0(H_{m}^{\Lambda})$ of finitely generated projective modules.  The isomorphism  $R_{\nu}^{\Lambda} \cong H_{\nu}^{\Lambda}$ induces a $\Z$-grading on blocks of cyclotomic Hecke algebras.   Brundan and Kleshchev use this grading to prove the cyclotomic quotient conjecture for type $A$. This can be viewed as a graded version of Ariki's categorification theorem.  A generalization of this conjecture to any simply-laced type should follow from the work of Varagnolo and Vasserot~\cite{VV} and the combinatorics of crystal graphs.

Brundan and Kleshchev's $\Z$-grading on blocks of cyclotomic Hecke algebras gives rise to a new grading on blocks of the symmetric group, enabling the study of graded representations of the symmetric group~\cite{BKW,BK2} and the construction of graded Specht modules~\cite{BKW}.  We also remark that prior to Brundan and Kleshchev's work, Brundan and Stroppel~\cite{BS} established the cyclotomic quotient conjecture for level two representations at $q=1$ in type $A_{\infty}$.

Even with Brundan and Kleshchev's proof of the cyclotomic quotient conjecture in type $A$, it is still difficult to construct an explicit basis for cyclotomic quotients $R_{\nu}^{\Lambda}$.  Brundan and Kleshchev's proof of the cyclotomic quotient conjecture utilizes the isomorphism $R_{\nu}^{\Lambda} \cong H_{\nu}^{\Lambda}$.  However, this isomorphism is rather sophisticated and does not directly lead to an explicit homogeneous basis for $R_{\nu}^{\Lambda}$ in type $A$.  For example, Brundan and Kleshchev conjecture~\cite[Conjecture 2.3]{BK1} that for type $A_{\infty}$ the nilpotency of the generator $x_{r,\ii}$ is less than or equal to the level $\ell(\Lambda)$.

In this note we define an upper bound $b_r=b_r(\ii)$, called the antigravity bound, for the nilpotency of the generator $x_{r,\ii}$ in $R_{\nu}^{\Lambda}$.  We prove by induction that $x_{r,\ii}^{b_r}=0$. Our upper bound implies Brundan and Kleshchev's nilpotency conjecture since $b_r$ is always less than or equal to the level $\ell(\Lambda)$. Methods used in our proof may be relevant for determining the nilpotency degrees for generators $x_{r,\ii}$ in other types.  We hope that understanding these nilpotency degrees will be a step towards constructing explicit homogeneous monomial bases for these quotients.

The bound is most naturally understood using the combinatorial device of `bead and runner' diagrams used by Kleshchev and Ram~\cite{KR} in their study of homogeneous representations of rings $R(\nu)$.  Kleshchev and Ram give a way to turn a sequence $\ii \in \seq(\nu)$ into a configuration of numbered beads on runners coloured by the vertices of $\Gamma$.  The main idea of our proof is to study bead and runner diagrams in `anti-gravity'.

To prove the induction step we show that either the nilpotency of $x_{m, \ii}$ can be determined from the nilpotency of some $x_{m',\ii'}$ where $m' < m$, $|\ii'| < |\ii|$, and $b_{m'}(\ii')=b_m(\ii)$, or the sequence $\ii$ has a special form.  Sequences $\ii$ with this special form are called stable antigravity sequences and they are characterized in terms of bead and runner diagrams associated to the sequence $\ii$.  For stable antigravity sequences we prove directly that the antigravity bound holds.

For the readers convenience we include a table summarizing the notation used by the second author in collaboration with Khovanov and the notation used by Brundan and Kleshchev.  In this note we write $\Lambda=\sum_i \lambda_i \cdot i \in \N[I]$ for a dominant integral weight, and we write the corresponding cyclotomic quotient as $R_{\nu}^{\Lambda}$.
\medskip

\begin{tabular}{|c|c|c|}
  \hline
 Description &{\bf Brundan-Kleshchev} & {\bf Khovanov-Lauda} \\ \hline \hline
 graph, vertex set & $\Gamma, I$  &  same  $\xy (0,4)*{}; (0,-3)*{};\endxy$\\
 \hline
 lattices indexed by $I$ & $P:=\bigoplus_{i \in I} \Z \Lambda_i$, $Q:= \bigoplus_{i \in I} \Z \alpha_i$ &  $\Z[I]$ $\xy (0,4)*{}; (0,-3)*{};\endxy$
 \\ \hline
 positive root& $\alpha\in Q^+$  &
 $\nu=\sum_{i \in I} \nu_i \cdot i \in \N[I]$ $\xy (0,4)*{}; (0,-3)*{};\endxy$ $\xy (0,4)*{}; (0,-3)*{};\endxy$\\
  \hline
  set of sequences & $I^{\alpha}:= \scs \{\ii = (i_1, \dots, i_d)|\alpha_{i_1} + \dots + \alpha_{i_d}=\alpha \}$ & $\seq(\nu)$ $\xy (0,4)*{}; (0,-3)*{};\endxy$\\
  \hline
  $\txt{length of sequence\\ $\ii=i_1i_2\cdots i_m$}$ & ${\rm ht}(\alpha)=\sum_{i \in I}(\Lambda_i,\alpha)$ &
  $|\nu| = \sum_{i \in I} \nu_i$ $\xy (0,6)*{}; (0,-7)*{};\endxy$
  \\ \hline
  idempotents & $e(\ii)$ & $1_{\ii}$$\xy (0,4)*{}; (0,-3)*{};\endxy$ \\ \hline
  \txt{dot on $r$th strand \\of sequence $\ii$} & $y_r e(\ii)$  & $x_{r,\ii}$ $\xy (0,4)*{};
  (0,-3)*{};\endxy$\\ \hline
  \txt{crossing of $r$th and $r+1$st \\strand of sequence $\ii$} & $\psi_r e(\ii)$
  & $\delta_{r,\ii}$ $\xy (0,6)*{}; (0,-6)*{};\endxy$\\ \hline
   rings and quotients & $R_{\alpha}$, $R_{\alpha}^{\Lambda}$  & $R(\nu)$,
   $R(\nu,\lambda)$ $\xy (0,4)*{}; (0,-3)*{};\endxy$\\
   \hline
\end{tabular}

\bigskip

\noindent {\it Acknowledgments:}  We thank Mikhail Khovanov and Alexander Kleshchev for valuable discussions.   We also thank Ben Elias for comments on a previous version.  AH was supported by the National Science Foundation under Grant No.\ 0653646.  AL was partially supported by the NSF grants  DMS-0739392
and DMS-0855713.

%
\section{Quotients in type $A_{\infty}$}
%

Consider the quiver $\Gamma$ of type $A_{\infty}$, where we identify the
vertex set $I$ with $\Z$:
\begin{equation}
  \Gamma =    \qquad \quad  \xy
  (-35,0)*+{-1}="-1";
  (-25,0)*+{0}="0";
  (-15,0)*+{1}="1";
  (-5, 0)*+{2}="2";
  (5,  0)*+{3}="3";
  (-45,0); "-1" **\dir{-};"-1";"0" **\dir{-}; "0";"1" **\dir{-};
  "1";"2" **\dir{-};"2";"3" **\dir{-};"3";(15,0) **\dir{-};
  (20,0)*{\cdots };
  (-50,0)*{\cdots };
  \endxy \label{eq_gamma_gl}
\end{equation}
Vertex $i$ is connected by an edge to vertex $j$ if and only if $j=i\pm 1$.

%
\subsection{Bead and runner diagrams}
%

To a sequence $\ii = i_1\dots i_m$ and an elementary transposition $s_r$ in the symmetric group $S_m$ we can associate the crossing $\delta_{r,\ii}$ in $R(\nu)$.  A transposition $s_r$ is called an admissible transposition if the corresponding element $\delta_{r,\ii}$ in $R(\nu)$ has degree zero.  This happens when the crossing $\delta_{r,\ii}$ involves strands coloured by vertices not connected by an edge in $\Gamma$.  For $\nu \in \N[I]$ the weight graph $G_{\nu}$ has as its vertices all the sequences $\ii \in \seq(\nu)$.  Sequences $\ii$ and $\jj$ are connected by an edge in $G_{\nu}$ if $\ii = s_r(\jj)$ for an admissible transposition $s_r$.

\noindent\parbox{4.5in}{\quad \; We recall the parametrization of the connected components of $G_{\nu}$ due to Kleshchev and Ram~\cite[Section 2.5]{KR}.  The set $I\times R_{\geq 0}$ is called a $\Gamma$-abacus.  For a vertex $i\in I$, the subset $\{i\}\times \R_{\geq 0}$ is called a runner of the $\Gamma$-abacus, or the runner coloured by the vertex $i$.  We slide `beads', whose shape depends on $\Gamma$, onto the runners of the $\Gamma$-abacus and gravity pulls the beads down the runners creating a {\em bead and runner diagram}. On the right is an example for $D_5$ of a $\Gamma$-abacus with 3 beads on various runners.  Bead and runner diagrams can be understood in terms of heaps introduced by Viennot~\cite{Vie}.}
\qquad \quad\parbox{0.9in}{$ \xy
  (0,0)*{\includegraphics[scale=0.5]{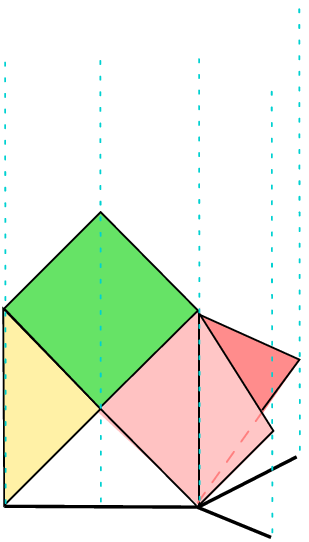}};
 \endxy$} \medskip

Fix  $\nu=\sum_{i \in I}\nu_i \cdot i \in \N[I]$ with $|\nu|=m$. A configuration $\lambda$ of type $\nu$ is obtained by placing $m$ beads on the runners with $\nu_i$ beads placed on the runner $i$ for each $i \in I$.  If $\lambda$ is a configuration of type $\nu$, then we write ${\rm Supp}(\lambda):={\rm Supp}(\nu)$, which can be thought of as those runners $i$ with at least one bead on them.  A $\lambda$-tableau is a bijection $T \maps \{1,2,\dots,m\} \to \{\text{beads of $\lambda$} \}$.  A bead is removable if it can be slid off its runner without interfering with other beads.  A {\em standard $\lambda$-tableau} is a special numbering of the beads: the largest numbered bead is removable, after removing this bead the next largest numbered bead is removable, and so on until all the beads are removed.  An example for $\Gamma$ an infinite chain appears on the left side of \eqref{eq_abacus}.

Given $\ii =i_1\dots i_m\in \seq(\nu)$ we define a standard $\lambda$-tableau $T^{\ii}$ by placing a bead labelled $1$ onto the runner coloured $i_1$, then a bead labelled $2$ onto the runner coloured $i_2$, and so on until the last bead labelled $m$ is placed onto runner coloured $i_m$.  The resulting configuration of beads on the abacus, disregarding the numbers labelling the beads,  is denoted by $\conf(\ii)$.  Given a standard $\lambda$-tableau $T$ we get a sequence $\ii^{T} = i_1 \dots i_m$ in $\seq(\nu)$, where $i_a \in I$ is the colour of the runner that the $a$th bead is on.

\begin{prop}[Kleshchev-Ram~\cite{KR},  Proposition 2.4] \label{prop_KR}
Two sequences $\ii$ and $\jj $ in $\seq(\nu)$ are in the same connected component of the weight graph $G_{\nu}$ if and only if $\conf(\ii) = \conf(\jj)$.  Moreover,  the assignments $\ii \mapsto T^{\ii}$ and $T \mapsto \ii^{T}$ are mutually inverse bijections between the set of standard $\lambda$-tableau and the set of all sequences $\ii$ in $\seq(\nu)$ with $\conf(\ii)=\lambda$.
\end{prop}

%
\subsection{Antigravity}
%

Bead and runner diagrams in type $A_{\infty}$ are closely related to the `Russian' notation for Young diagrams.  The advantage of `Russian' notation is that it takes `gravity' into account -- beads are pulled to the bottom of a bead and runner diagram.   In constructing our nilpotency bound `antigravity' will play an equally important role.

To study bead and runner diagrams in antigravity we choose a bead on the diagram and anchor it in place.  Rather than beads sliding down the abacus via gravity, beads not trapped below the anchored bead are pulled off the runners by {\em antigravity}.   In the example below, the box labelled by `13' is the anchored bead.
\begin{equation} \label{eq_abacus}
 \xy
  (0,0)*{\includegraphics[scale=0.8]{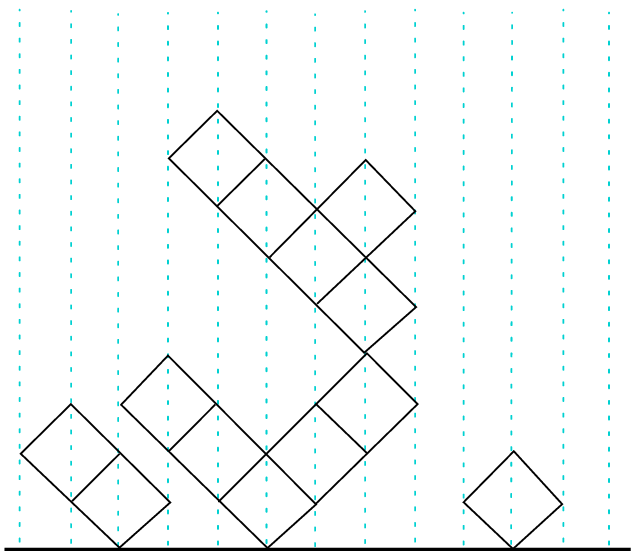}};
   (-4,-18)*{1}; (16,-18)*{7}; (-16,-18)*{3};
    (-20,-14)*{4}; (0,-14)*{2}; (-8,-14)*{5};
     (4,-10)*{6}; (-12,-10)*{11};  (4,-2)*{8};  (0,2)*{9};  (4,6)*{12}; (-4,6)*{10};
  (-8,10)*{13};
  (-21,-25)*{\scs -4};(-17,-25)*{\scs -3};(-13,-25)*{\scs -2}; (-9,-25)*{\scs -1};(-4,-25)*{\scs 0};(0,-25)*{\scs 1};(4,-25)*{\scs 2};(8,-25)*{\scs 3}; (12,-25)*{\scs 4}; (16,-25)*{\scs 5};
 \endxy
 \quad \xy {\ar^{\txt{antigravity}} (-9,0)*{}; (9,0)*{}}\endxy
  \xy
  (0,0)*{\includegraphics[scale=0.8]{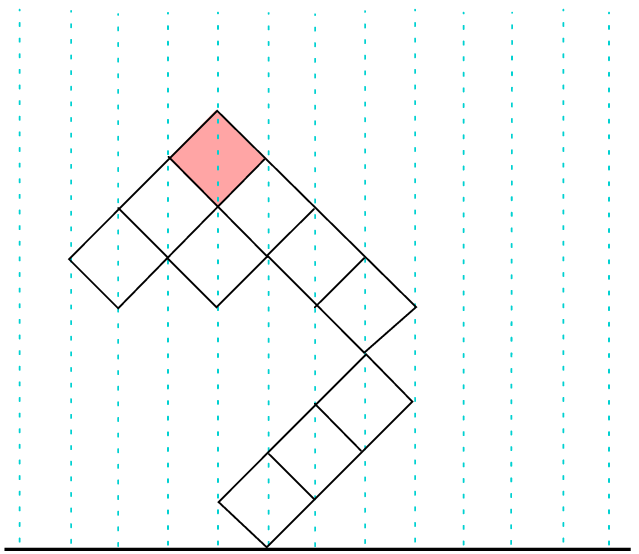}};
   (-4,-18)*{1}; (-16,2)*{3};
    (0,-14)*{2}; (-8,2)*{5};
     (4,-10)*{6}; (-12,6)*{11};  (4,-2)*{8};  (0,2)*{9};  (-4,6)*{10};
  (-8,10)*{\mathbf{13}};
    (-21,-25)*{\scs -4};(-17,-25)*{\scs -3};(-13,-25)*{\scs -2}; (-9,-25)*{\scs -1};(-4,-25)*{\scs 0};(0,-25)*{\scs 1};(4,-25)*{\scs 2};(8,-25)*{\scs 3}; (12,-25)*{\scs 4}; (16,-25)*{\scs 5};
 \endxy
\end{equation}
Boxes labelled `4', `7', and `12' have been slid off the abacus by antigravity.  Boxes labelled `3', `5', and `11' are slid up the abacus towards the anchored bead.

An {\em antigravity configuration} $a$ is a bead and runner diagram in antigravity for some choice of anchored bead.   We say that an antigravity configuration $a$ is of type $\nu=\sum_i \nu_i \cdot i$ if there are $\nu_i$ beads on runner $i$ in antigravity.  An antigravity configuration can be regarded as an ordinary configuration, also denoted $a$, by restoring ordinary gravity so that the remaining beads slide down the abacus.  Hence, for an antigravity configuration $a$ of type $\nu$, write ${\rm Supp}(a)= \{i \mid \nu_i \neq 0 \}$.  This is the same as ${\rm Supp}(a)$, where $a$ is regarded as an ordinary configuration.


\paragraph{Antigravity moves:}  Given a configuration of beads on a bead and runner diagram, considered in antigravity for some fixed bead, the following moves alter the antigravity configuration of the beads.
\begin{enumerate}[1)]
 \item square move:
$ \xy
  (0,0)*{\includegraphics[scale=0.75]{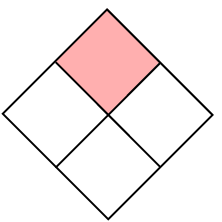}};
 \endxy \quad \mapsto \quad
  \xy
  (0,0)*{\includegraphics[scale=0.75]{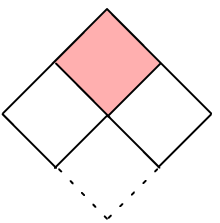}};
 \endxy \quad \mapsto \quad  \xy
  (0,0)*{\includegraphics[scale=0.75]{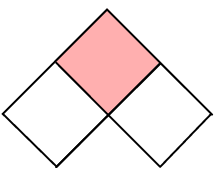}};
 \endxy$

The shaded box indicates the anchor.  This move removes the lower bead in the square configuration and is only applied when the top box in the square is the anchor.

 \item stack move:
$  \xy
  (0,0)*{\includegraphics[scale=0.75]{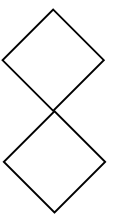}};
 \endxy \quad \mapsto \quad  \xy
  (0,0)*{\includegraphics[scale=0.75]{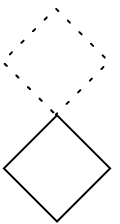}};
 \endxy\quad \mapsto \quad  \xy
  (0,0)*{\includegraphics[scale=0.75]{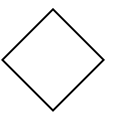}};
 \endxy$

The stack move is applicable only when there are no beads in between the two stacked beads.  In the diagram the top bead is destroyed without affecting other beads.  After applying this move, beads not held in place by the anchored bead slide freely up the abacus in antigravity.

  \item  $L$-move: the $L$-move destroys the lowest box in an L-like configuration:
  $$ \xy
  (0,0)*{\includegraphics[scale=0.75]{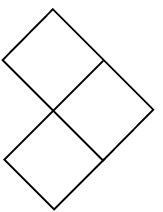}};
 \endxy \quad \mapsto \quad
 \xy
  (0,0)*{\includegraphics[scale=0.75]{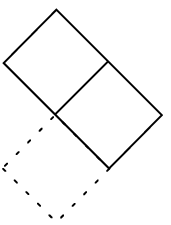}};
 \endxy \quad \mapsto \quad \xy
  (0,0)*{\includegraphics[scale=0.75]{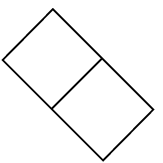}};
 \endxy  \qquad \text{or}  \qquad  \xy
  (0,0)*{\includegraphics[angle=180, scale=0.75]{anti2.eps}};
 \endxy \quad \mapsto \quad  \xy
  (0,0)*{\reflectbox{\includegraphics[scale=0.75]{anti3g.eps}}};
 \endxy \quad \mapsto \quad  \xy
  (0,0)*{\reflectbox{\includegraphics[scale=0.75]{anti3.eps}}};
 \endxy $$
After applying this move, beads slide freely up the abacus in antigravity.
\end{enumerate}

A configuration of beads stable under antigravity and the antigravity moves is called a {\em stable antigravity configuration}, or a {\em stable configuration}.


While square moves that do not involve the anchor are not directly reducible using the antigravity moves, for any such square configuration to exist in an antigravity configuration there must also be a square configuration that does involve the anchor.  After simplifying this anchor square move, an $L$-like configuration will be created.  Applying antigravity moves and iterating this process will then simplify the non-anchor square move.  It is easy to see that:
\begin{prop}
For type $A_{\infty}$ with vertex set $I$ identified with $\Z$, a stable antigravity configuration is any antigravity configuration with exactly one bead on the runner $i$ for each $i$ in an interval $[a,b]$ containing the anchored bead.
\end{prop}
Several examples are shown below where the anchored bead is shaded:
\[
\xy
  (0,0)*{\includegraphics[scale=0.7]{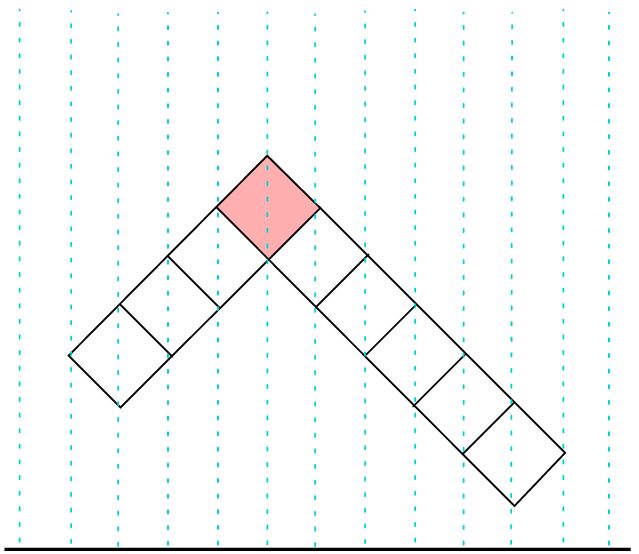}};
  (-16,-22)*{\scs -2}; (-11.5,-22)*{\scs -1};(-7,-22)*{\scs 0}; (-4,-22)*{\scs 1};
  (0,-22)*{\scs 2};(3.5,-22)*{\scs 3};(6.5,-22)*{\scs 4};(10,-22)*{\scs 5}; (13.5,-22)*{\scs 6}; (17.5,-22)*{\scs 7};
 \endxy
 \qquad
 \xy
  (0,0)*{\includegraphics[scale=0.7]{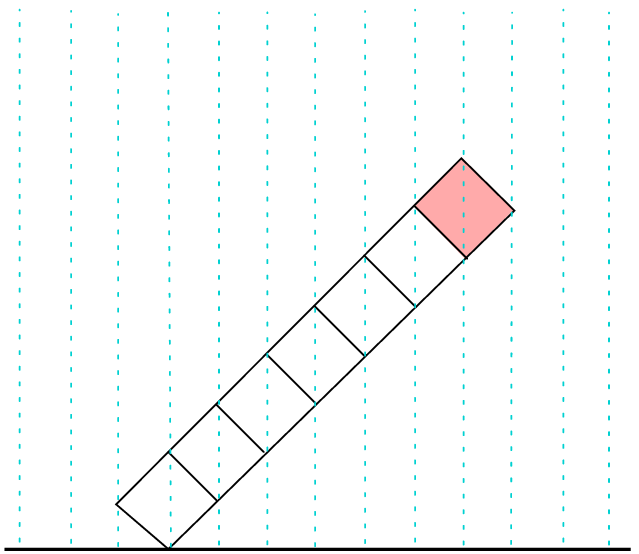}};
  (-16,-22)*{\scs -5}; (-12,-22)*{\scs -4};(-8,-22)*{\scs -3}; (-4,-22)*{\scs -2};
  (0,-22)*{\scs -1};(3.5,-22)*{\scs 0};(6.5,-22)*{\scs 1};(10,-22)*{\scs 2}; (13.5,-22)*{\scs 3}; (17.5,-22)*{\scs 4};
 \endxy
 \qquad
 \xy
  (0,0)*{\includegraphics[scale=0.7]{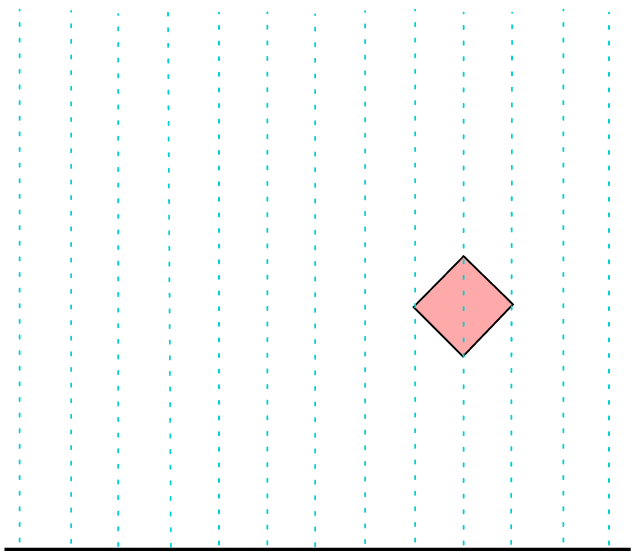}};
  (-16,-22)*{\scs -4}; (-12,-22)*{\scs -3};(-8,-22)*{\scs -2}; (-4,-22)*{\scs -1};
  (0,-22)*{\scs 0};(3.5,-22)*{\scs 1};(6.5,-22)*{\scs 2};(10,-22)*{\scs 3}; (13.5,-22)*{\scs 4}; (17.5,-22)*{\scs 5};\endxy
\]
The antigravity configuration in the first example is supported on $[-2,6]$, the second on $[-4,2]$, and the last configuration is supported on just a single vertex $[3]$.

\begin{defn}
Given a sequence $\ii=i_1 \dots i_r \dots i_m$,  anchor the bead corresponding to $i_r$ in $\conf(\ii)$.  Apply antigravity moves to the resulting configuration until the diagram stabilizes. From the beads that remain we form the {\em $r$-stable antigravity configuration $a_r(\ii)$ of $\ii$}, or $r$-stable configuration of $\ii$.
\end{defn}

It is easy to see that ${\rm Supp}(a_r(\ii))$ is completely determined by the support of the antigravity configuration of $\ii$ with anchor $i_r$, since after turning on antigravity all antigravity moves preserve the support of the configuration. Thus, the antigravity moves simply remove beads until there is exactly one bead on each runner in the support, so that $a_r(\ii)$ is well defined and independent of the order in which antigravity moves are applied.

A sequence $\ii$ is called {\em $r$-stable} if the configuration of $\ii$ in antigravity with anchored bead $i_r$ is the same as $a_r(\ii)$.

\begin{defn}
Let $\Lambda=\sum_{i \in I} \lambda_i \cdot i \in \N[I]$ and $\ii = i_1 \dots i_m$. For any $1 \leq r \leq m$ define the {\em $r$-antigravity bound of $\ii$} as
\[
 b_r=b_r(\ii) := \sum_{j \in {\rm Supp}(a_r(\ii))} \lambda_j.
\]
\end{defn}

\begin{example}
For the sequence $\ii=(0,1,-3,-4,-1,2,5,2,1,0,-2,2,-1)$, we compute the $13$-stable antigravity configuration $a_{13}(\ii)$ at $i_{13}=-1$ as follows:
\[\qquad\qquad\quad
 \xy
  (0,0)*{\includegraphics[scale=0.8]{sample-config.eps}};
   (-4,-18)*{1}; (16,-18)*{7}; (-16,-18)*{3};
    (-20,-14)*{4}; (0,-14)*{2}; (-8,-14)*{5};
     (4,-10)*{6}; (-12,-10)*{11};  (4,-2)*{8};  (0,2)*{9};  (4,6)*{12}; (-4,6)*{10};
  (-8,10)*{13};
  (-21,-25)*{\scs -4};(-17,-25)*{\scs -3};(-13,-25)*{\scs -2}; (-9,-25)*{\scs -1};(-4,-25)*{\scs 0};(0,-25)*{\scs 1};(4,-25)*{\scs 2};(8,-25)*{\scs 3}; (12,-25)*{\scs 4}; (16,-25)*{\scs 5};(20,-25)*{\scs 6};
 \endxy
 \quad \xy {\ar^{\txt{antigravity}} (-9,0)*{}; (9,0)*{}}\endxy
\quad  \xy
  (0,0)*{\includegraphics[scale=0.8]{sample-grconfig.eps}};
   (-4,-18)*{1}; (-16,2)*{3};
    (0,-14)*{2}; (-8,2)*{5};
     (4,-10)*{6}; (-12,6)*{11};  (4,-2)*{8};  (0,2)*{9};  (-4,6)*{10};
  (-8,10)*{\mathbf{13}};
    (-21,-25)*{\scs -4};(-17,-25)*{\scs -3};(-13,-25)*{\scs -2}; (-9,-25)*{\scs -1};(-4,-25)*{\scs 0};(0,-25)*{\scs 1};(4,-25)*{\scs 2};(8,-25)*{\scs 3}; (12,-25)*{\scs 4}; (16,-25)*{\scs 5};(20,-25)*{\scs 6};
 \endxy
\]
\[
 \quad \xy {\ar^{\xy
  (0,0)*{\includegraphics[scale=0.5]{square-move.eps}};
  (0,-2.5)*{\scs 5};  (0,2.5)*{\scs 13};(-2.5,0)*{\scs 11};(2.5,0)*{\scs 10}
 \endxy} (-9,0)*{}; (9,0)*{}}\endxy
  \xy
  (0,0)*{\includegraphics[scale=0.8]{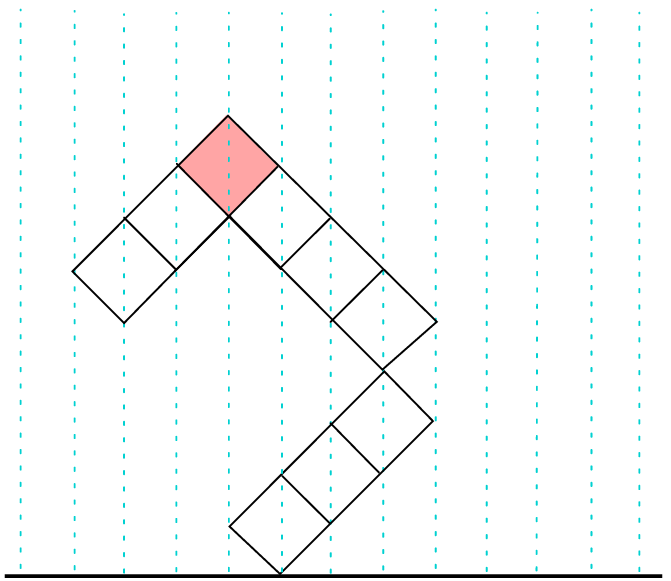}};
   (-4.5,-19)*{1}; (-17,2)*{3};
    (0,-14.5)*{2}; (4,-10.5)*{6};
    (-13,6.3)*{11}; (4,-2.5)*{8};
    (0,2)*{9};  (-4.5,6.3)*{10};
  (-8.7,10.5)*{\mathbf{13}};
    (-22,-25)*{\scs -4};(-18,-25)*{\scs -3};(-14,-25)*{\scs -2}; (-9,-25)*{\scs -1};(-4,-25)*{\scs 0};(0,-25)*{\scs 1};(4,-25)*{\scs 2};(8,-25)*{\scs 3}; (12,-25)*{\scs 4}; (16,-25)*{\scs 5};(20,-25)*{\scs 6};
 \endxy
 \quad \xy {\ar^{\xy
  (0,0)*{\includegraphics[scale=0.5]{anti1.eps}}; (0,-2.5)*{\scs 6};
  (0,2.5)*{\scs 8};
 \endxy}_{\txt{antigravity}} (-9,0)*{}; (9,0)*{}}\endxy
   \xy
  (0,0)*{\includegraphics[scale=0.8]{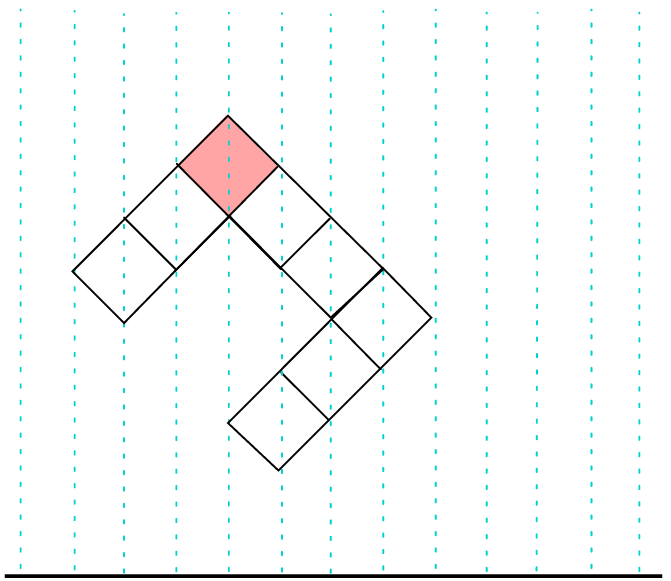}};
    (-4.5,-10.5)*{1}; (-17,2)*{3};
    (0,-6.2)*{2};
    (-13,6.3)*{11}; (4,-2.2)*{6};
    (0,2)*{9};  (-4.5,6.3)*{10};
  (-8.7,10.5)*{\mathbf{13}};
    (-22,-25)*{\scs -4};(-18,-25)*{\scs -3};(-14,-25)*{\scs -2}; (-9,-25)*{\scs -1};(-4,-25)*{\scs 0};(0,-25)*{\scs 1};(4,-25)*{\scs 2};(8,-25)*{\scs 3}; (12,-25)*{\scs 4}; (16,-25)*{\scs 5};(20,-25)*{\scs 6};
 \endxy
 \]
\[
 \quad \xy {\ar_{\txt{antigravity}}^{\xy
  (0,0)*{\includegraphics[scale=0.5]{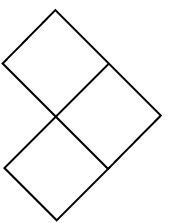}};
  (-1.5,-2.7)*{\scs 2}; (-1.5,2.5)*{\scs 9};
  (1.4,0)*{\scs 6}; \endxy}(-9,0)*{}; (9,0)*{}}\endxy
   \xy
  (0,0)*{\includegraphics[scale=0.8]{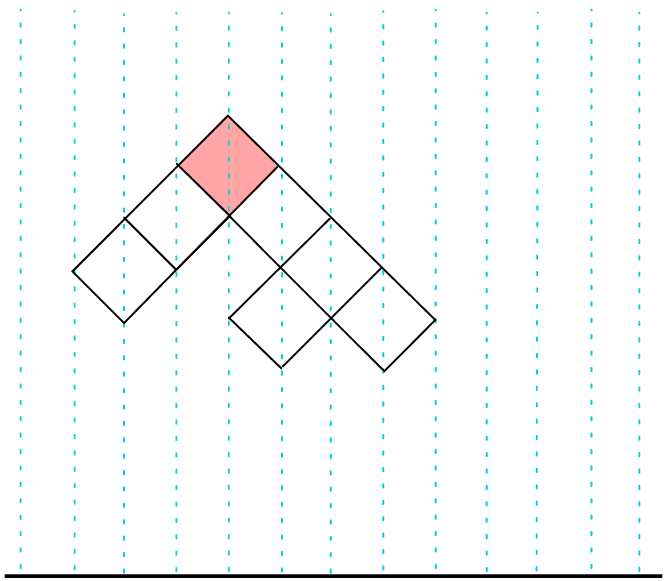}};
   (-4.5,-2.2)*{1}; (-17,2)*{3};
    (-13,6.3)*{11}; (4,-2.2)*{6};
    (0,2)*{9};  (-4.5,6.3)*{10};
  (-8.7,10.5)*{\mathbf{13}};
    (-22,-25)*{\scs -4};(-18,-25)*{\scs -3};(-14,-25)*{\scs -2}; (-9,-25)*{\scs -1};(-4,-25)*{\scs 0};(0,-25)*{\scs 1};(4,-25)*{\scs 2};(8,-25)*{\scs 3}; (12,-25)*{\scs 4}; (16,-25)*{\scs 5};(20,-25)*{\scs 6};
 \endxy
 \quad \xy {\ar^{\xy
  (0,0)*{\includegraphics[scale=0.5]{L.eps}};
  (-1.5,-2.7)*{\scs 1}; (-1.5,2.5)*{\scs 10};(1.4,0)*{\scs 9};\endxy}(-9,0)*{}; (9,0)*{}}\endxy
   \xy
  (0,0)*{\includegraphics[scale=0.8]{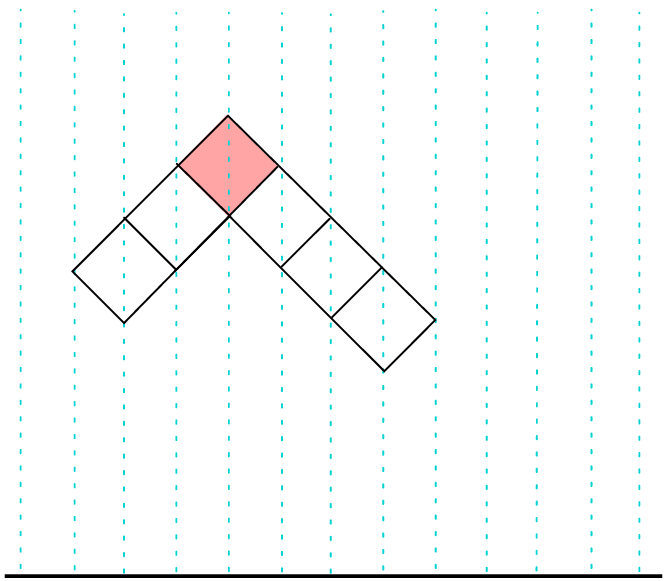}};
(-17,2)*{3};
    (-13,6.3)*{11}; (4,-2.2)*{6};
    (0,2)*{9};  (-4.5,6.3)*{10};
  (-8.7,10.5)*{\mathbf{13}};
    (-22,-25)*{\scs -4};(-18,-25)*{\scs -3};(-14,-25)*{\scs -2}; (-9,-25)*{\scs -1};(-4,-25)*{\scs 0};(0,-25)*{\scs 1};(4,-25)*{\scs 2};(8,-25)*{\scs 3}; (12,-25)*{\scs 4}; (16,-25)*{\scs 5};(20,-25)*{\scs 6};
 \endxy
 \]
The $13$-stable configuration for $\ii$ has ${\rm Supp}(a_{13}(\ii))=\{-3,-2,-1,0,1,2\}$.  In this example we could have also performed the `stack move' on boxes labelled `6' and `8', then applied several other antigravity moves.  The end result is the same.  All beads below the highest bead on each runner in ${\rm Supp}(a_{13}(\ii))$ are removed by the antigravity moves.  The 13-antigravity bound of $\ii$ is $b_{13} =  \sum_{j=-3}^2\lambda_{j}$.
\end{example}

\begin{prop} \label{prop_subset}
For $\ii \in \seq(\nu)$ let $\ii'$ denote the subsequence obtained from $\ii$ by removing those terms corresponding to beads that are pulled off the bead and runner diagram in antigravity with anchor $i_m$.  If $\jj$ is any subsequence of $\ii'$ and $i_r \in \jj$, then $a_r(\jj) \subset a_m(\ii)$ and in particular $b_r(\jj) \leq b_m(\ii)$.
\end{prop}

\begin{proof}
Recall that $a_m(\ii)$ is determined by the support of $\conf(\ii)$ in antigravity with anchor $i_m$.  That is, ${\rm Supp}(a_m(\ii))={\rm Supp}(\conf(\ii'))$, where $\ii'$ is as above.  If $\jj$ is any subsequence of $\ii'$ and $i_r \in \jj$, then the support of the configuration $\conf(\jj)$ considered in antigravity with anchor $i_r$ must be contained in ${\rm Supp}(\conf(\ii))$.  Hence, $a_r(\jj) \subset a_m(\ii)$ and the result follows.
\end{proof}

\begin{rem} \hfill
\begin{itemize}
 \item
The antigravity bound only depends on the shape of a configuration $\lambda$, not on the entries that appear in a given $\lambda$-tableau.  In particular, if $\conf(\ii)=\conf(\jj)$ for some sequences $\ii$ and $\jj$, then by Kleshchev and Ram's characterization of configurations (see Proposition~\ref{prop_KR}) we must have $\jj=s(\ii)$ for some permutation $s=s_{j_1}\dots s_{j_k}$ with each $s_{j_a}$ an admissible transposition. It is clear that the $r$-stable configuration and $r$-antigravity bound for $\ii$ are the same as the $s(r)$-stable configuration and $s(r)$-antigravity bound for $\jj$.

 \item The $r$-stable antigravity sequence for $\ii = i_1 \dots i_r \dots i_{m}$ does not depend on the terms of $\ii$ that occur after $i_r$.  In particular, all beads corresponding to terms in the subsequence $i_{r+1}\dots i_{m}$ are removed from the diagram when antigravity is turned on.  Hence, if $\ii'=i_1 \dots i_r$, then the $r$-stable antigravity configurations for these two sequences $\ii$ and $\ii'$ are the same, $a_r(\ii)=a_r(\ii')$.
\end{itemize}
\end{rem}

%
\subsection{Local relations for cyclotomic quotients}
%

The relations in $R(\nu)$ for identically coloured strands imply
\begin{equation} \label{eq_crossdouble}
 \xy   (0,0)*{\twocross{i}{i}}; (-2,2)*{\bullet};\endxy \quad = \quad
 \xy (0,0)*{\dcross{i}{i}};  \endxy
 \qquad \qquad
 \xy   (0,0)*{\twocross{i}{i}}; (2,2)*{\bullet};\endxy  \quad = \quad
 - \;\xy (0,0)*{\dcross{i}{i}};  \endxy
\end{equation}
and for $b>0$
\begin{eqnarray}
 \xy  (0,0)*{\dcrossul{i}{i}};(-5,3)*{\scs b}; \endxy \;\; - \;\;\xy (0,0)*{\dcrossdr{i}{i}}; (5,0)*{\scs b}; \endxy
 \;\;  = \;\;  \xy  (0,0)*{\dcrossdl{i}{i}};(-5,0)*{\scs b}; \endxy
 \;\;-\;\; \xy (0,0)*{\dcrossur{i}{i}}; (5,2)*{\scs b}; \endxy
 \;\;  = \;\;
    \sum_{\ell_1+\ell_2=b-1} \xy (-3,0)*{\supdot{i}};
    (3,0)*{\supdot{i}}; (7,2)*{\scs \ell_2};(-7,2)*{\scs \ell_1}; \endxy
 \label{eq_inddotslide}
\end{eqnarray}
Recall that $i^s := i i \dots i$ where vertex $i$ appears $s$ times.

\begin{prop} \label{prop_symcomb}
Let $\ii = \ii' i^{s} \ii'' \in \seq(\nu)$, $s\geq 1$,  with $|\ii'|=r$. If
$x_{r+1,\ii}^a=0$, then
\begin{equation}
  \sum_{\ell_1 +  \dots + \ell_{s} = a-(s-1)} x_{r+1,\ii}^{\ell_1}\cdot
  x_{r+2,\ii}^{\ell_2}\cdot
  \ldots \cdot x_{r+s,\ii}^{\ell_s} \;\; = \;\; 0.
\end{equation}
\end{prop}

\begin{proof}
The proof is by induction on the length $s$ of consecutive strands labelled $i$.
The base case is trivial. Assume the result holds up to length $s$, we will show
it also holds for length $s+1$.  Working locally around the $s+1$ consecutive strands labelled $i$
\begin{equation}
  \sum_{ \xy (0,0)*{\scs \ell_1+ \dots +\ell_{s+1}}; (0,-3)*{\scs =a-s}; \endxy}
  \xy (-3,0)*{\supdot{i}};
    (3,0)*{\supdot{i}}; (0,2)*{\scs \ell_1};(6,2)*{\scs \ell_2};
    (12,0)*{\cdots};
    (18,0)*{\supdot{i}}; (21,2)*{\scs \ell_{s}};
    (24,0)*{\supdot{i}}; (29,2)*{\scs \ell_{s+1}};
    \endxy \quad \refequal{\eqref{eq_iislide2}} \quad
    \sum_{ \xy (0,0)*{\scs \ell_1+ \dots +\ell_{s+1}}; (0,-3)*{\scs =a-s}; \endxy}
    \left(
   \xy   (-3,0)*{\supdot{i}}; (0,2)*{\scs \ell_1};
    (4,0)*{\cdots};
    (12,0)*{\dcrossdl{i}{i}};
    (9.5,3.2)*{\bullet}; (14.5,3.2)*{\bullet};
    (7,3,2)*{\scs \ell_s};(19,3,2)*{\scs \ell_{s+1}};\endxy
    \;\; - \;\;
     \xy   (-3,0)*{\supdot{i}}; (0,2)*{\scs \ell_1};
    (4,0)*{\cdots};
    (12,0)*{\dcross{i}{i}};
    (9.5,3.2)*{\bullet}; (13.2,2.2)*{\bullet};
    (15.5,4.5)*{\bullet};
    (7,3,2)*{\scs \ell_s};(17.5,1.5)*{\scs \ell_{s+1}};\endxy \right) \label{eq_is}
\end{equation}
Fixing the value $a-\ell:=\ell_1+\ell_2+\dots+\ell_{s-1}$, there is a symmetric
combination of $\ell$ dots on the last two strands, so we can write
\begin{equation} \label{eq_aminusl}
\sum_{\ell_s+\ell_{s+1} = \ell} \xy
 (0,0)*{\dcrossul{i}{i}}; (2.5,3.2)*{\bullet};
    (7,3,2)*{\scs \ell_{s+1}}; (-6,3,2)*{\scs \ell_{s}};
\endxy
\;\; \refequal{\eqref{eq_inddotslide}} \;\;
   \xy   (0,0)*{\twocross{i}{i}}; (-2,2)*{\bullet};(-6,3)*{\scs \ell+1};\endxy
   -
   \xy   (0,0)*{\twocross{i}{i}}; (2,9)*{\bullet};(7,9)*{\scs \ell+1};\endxy
\;\; \refequal{\eqref{eq_UUzero}} \;\; \xy   (0,0)*{\twocross{i}{i}};
(-2,2)*{\bullet};(-6,3)*{\scs \ell+1};\endxy
\end{equation}
Then \eqref{eq_is} can be written as
\begin{equation}
  \qquad \refequal{\eqref{eq_aminusl}} \quad \sum_{\ell_1+ \dots +\ell_{s-1}=a-s} \left(
  \xy
    (-3,0)*{\updot{i}};(3,0)*{\updot{i}};
    (0,2)*{\scs \ell_1};(6,2)*{\scs \ell_2};
    (12,0)*{\cdots};
    (20,0)*{\twocross{i}{i}};
    (18,2)*{\bullet};(18,-6)*{\bullet};(14,3)*{\scs \ell+1};
    \endxy
    -
      \xy
    (-3,0)*{\updot{i}};(3,0)*{\updot{i}};
    (0,2)*{\scs \ell_1};(6,2)*{\scs \ell_2};
    (12,0)*{\cdots};
    (20,0)*{\twocross{i}{i}};
    (18,2)*{\bullet};(14,3)*{\scs \ell+1}; (22.3,9)*{\bullet};
    \endxy \right) \nn
\end{equation}
for $\ell=a-(\ell_1+\cdots + \ell_{s-1})$.  If we write $\ell'_s=\ell+1$ and add
terms for $\ell'_s=0$, which are zero by $\eqref{eq_UUzero}$, then
\begin{equation}
 \sum_{ \xy (0,0)*{\scs \ell_1+ \dots +\ell_{s+1}}; (0,-3)*{\scs =a-s}; \endxy}
  \xy (-3,0)*{\updot{i}};
    (3,0)*{\updot{i}}; (0,2)*{\scs \ell_1};(6,2)*{\scs \ell_2};
    (12,0)*{\cdots};
    (18,0)*{\updot{ i}}; (21,2)*{\scs \ell_{s}};
    (24,0)*{\updot{ i}}; (29,2)*{\scs \ell_{s+1}};
    \endxy \;\; = \;\;
   \sum_{ \xy (0,0)*{\scs \ell_1+ \dots +\ell_{s-1}+\ell'_s}; (0,-3)*{\scs =a-(s-1)}; \endxy}
   \left(
  \xy
    (-3,0)*{\updot{ i}};(3,0)*{\updot{\; i}};
    (0,2)*{\scs \ell_1};(6,2)*{\scs \ell_2};
    (12,0)*{\cdots};
    (20,0)*{\twocross{i}{i}};
    (18,2)*{\bullet};(18,-6)*{\bullet};(15.5,3)*{\scs \ell'_s};
    \endxy
    -
      \xy
    (-3,0)*{\updot{\; i}};(3,0)*{\updot{\; i}};
    (0,2)*{\scs \ell_1};(6,2)*{\scs \ell_2};
    (12,0)*{\cdots};
    (20,0)*{\twocross{i}{i}};
    (18,2)*{\bullet};(15.5,3)*{\scs \ell'_s}; (22.3,9)*{\bullet};
    \endxy \right) \nn
\end{equation}
and both terms on the right are zero by the induction hypothesis.
\end{proof}

The following Proposition appears in an algebraic form in the work of Brundan and Kleshchev~\cite{BK1}.

\begin{prop}\label{prop_anchorc}
Consider the sequence $\ii = i_1 i_2  \dots i_m \in \seq(\nu)$ in $R_{\nu}^{\Lambda}$. If  $i_{m-1}=i_{m}$, then $x_{m-1,\ii}^{b}=0$ implies $x_{m,\ii}^{b}=0$.
\end{prop}

\begin{proof}
We work locally around the two identically coloured strands.  Using that $b$ dots
on the $(m-1)$st strand is zero we have for any $a \geq b$
\begin{eqnarray} \nn
 \xy (0,0)*{\dcrossdr{\;\;i_m}{\;\;i_m}}; (5,0)*{\scs a}; \endxy
   &\refequal{\eqref{eq_inddotslide}}& \xy  (0,0)*{\dcrossul{\;\;i_m}{\;\;i_m}};(-5,3)*{\scs a}; \endxy
 \quad - \quad
    \sum_{\ell_1+\ell_2=a-1} \xy (-3,0)*{\supdot{\;\;i_m}};
    (3,0)*{\supdot{\;\;i_m}}; (7,2)*{\scs \ell_2};(-7,2)*{\scs \ell_1}; \endxy
 \quad  = \quad   0\;-\sum_{\ell_1+\ell_2=a-1}
 \xy (-3,0)*{\supdot{\;\;i_m}}; (3,0)*{\supdot{\;\;i_m}}; (7,2)*{\scs \ell_2};(-7,2)*{\scs \ell_1}; \endxy
 \\
  &=& \xy  (0,0)*{\dcrossdl{\;\;i_m}{\;\;i_m}};(-5,0)*{\scs a}; \endxy
  \quad - \quad \sum_{\ell_1+\ell_2=a-1}
     \xy (-3,0)*{\supdot{\;\;i_m}}; (3,0)*{\supdot{\;\;i_m}}; (7,2)*{\scs \ell_2};(-7,2)*{\scs \ell_1};
     \endxy
 \quad \refequal{\eqref{eq_inddotslide}}\quad \xy (0,0)*{\dcrossur{\;\;i_m}{\;\;i_m}}; (5,2)*{\scs a}; \endxy
 \label{eq_crossbslide}
\end{eqnarray}
which implies
\begin{equation} \label{eq_crossbzero}
  \xy (0,0)*{\dcrossur{\;\;i_m}{\;\;i_m}}; (5,2)*{\scs b}; \endxy \quad \refequal{\eqref{eq_crossdouble}} \quad
 - \;  \xy   (0,0)*{\twocross{\;\;i_m}{\;\;i_m}}; (5,8)*{\scs b}; (2,2)*{\bullet};(2,9)*{\bullet};\endxy
 \quad \refequal{\eqref{eq_crossbslide}} \quad
  - \;  \xy   (0,0)*{\twocross{\;\;i_m}{\;\;i_m}}; (7,3)*{\scs b+1}; (2,2)*{\bullet};\endxy
  \quad \refequal{\eqref{eq_crossbslide}} \quad
   - \;  \xy   (0,0)*{\twocross{\;\;i_m}{\;\;i_m}}; (7,8)*{\scs b+1}; (2,9)*{\bullet};\endxy
 \quad  \refequal{\eqref{eq_UUzero}} \quad0.
\end{equation}
The claim follows since
\begin{equation}
  \xy (-3,0)*{\sup{\; i_m}};
    (3,0)*{\supdot{\; i_m}}; (6,2)*{\scs b};\endxy \quad  \refequal{\eqref{eq_iislide2}} \quad
    \xy  (0,0)*{\dcrossdl{\;\;i_m}{\;\;i_m}};(3,3.7)*{\bullet};(6,3)*{\scs b}; \endxy \quad - \quad
    \xy (0,0)*{\dcrossur{\;\;i_m}{\;\;i_m}}; (7,3)*{\scs b+1}; \endxy \quad \refequal{\eqref{eq_crossbzero}} \quad 0.
\end{equation}
\end{proof}

%
\subsection{Factoring sequences}
%

Recall from \cite{KL} elements ${}_{\jj}1_{\ii}$ in $R(\nu)$.  They are represented by diagrams with the fewest number of crossings that connect the sequence $\ii$ to the sequence $\jj$.
For example,
\[
{_{jjiki}1_{ijkij}} \quad = \quad
 \xy
 (0,0)*{ \includegraphics[scale=0.4]{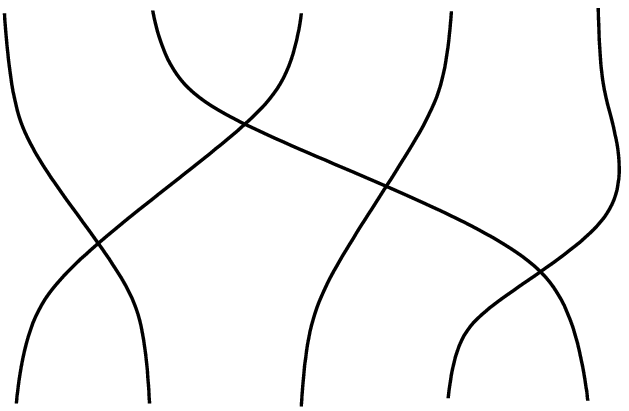}};
 (-12,-10)*{\scs i}; (-7,-10)*{\scs j};  (0,-10)*{\scs k}; (6,-10)*{\scs i}; (11.5,-10)*{\scs j};
 (-12.5,10)*{\scs j};  (-7,10)*{\scs j}; (0,10)*{\scs i}; (6,10)*{\scs k}; (12,10)*{\scs i};
 \endxy
\]
In particular, identically coloured strands do not intersect in ${}_{\jj}1_{\ii}$ and ${}_{\ii}1_{\ii}$ is just $1_{\ii}$.

Consider $\ii\in \seq(\nu)$ with $\ii = \ii' \ii''\ii'''$ and $\ii' \in \seq(\nu')$, $\ii'' \in \seq(\nu'')$,  $\ii''' \in \seq(\nu''')$, where $\nu = \nu'+\nu''+\nu'''$.  Write $R_{\ii',\ii'',\ii'''}$ for the image of $R(\nu')1_{\ii'} \otimes R(\nu'')1_{\ii''} \otimes R(\nu''')1_{\ii'''}$ in $R(\nu)$ under the natural inclusion
$
 R(\nu') \otimes R(\nu'') \otimes R(\nu''') \longrightarrow R(\nu).
$
We say that the sequence $\ii=\ii' i_r \ii'''$ has an {\em $r$-factorization} through the sequence $\jj$ if
\begin{equation}
 1_{\ii} = R_{\ii',i_r, \ii'''} \left( {}_{\ii}1_{\jj}\right)\left( {}_{\jj}1_{\ii}\right)R_{\ii',i_r, \ii'''}.
\end{equation}
More generally we say that $\ii$ has an $r$-factorization through a finite collection of sequences $\{\jj_a\}_a$, where some $\jj_a$ may be repeated, if
\[
 1_{\ii} = \sum_{a} \; R_{\ii',i_r, \ii'''} \left( {}_{\ii}1_{\jj_a}\right)\left( {}_{\jj_a}1_{\ii}\right)R_{\ii',i_r, \ii'''}.
\]

\begin{example}
The sequence $\ii$ has an $r$-factorization through sequence $s_r(\ii)$ for any admissible transposition $s_r$ since
\[
 \xy
    (-12,0)*{\up{i_1}};
    (-8,0)*{\cdots};
    (-3,0)*{\up{i_r}};
    (3,0)*{\up{i_{r+1}}};
    (8,0)*{\cdots};
    (12,0)*{\up{i_m}};
    \endxy
\;\; = \;\;
 \xy
    (-12,0)*{\up{i_1}};
    (-8,0)*{\cdots};
    (0,0)*{\twocross{i_r}{i_{r+1}}};
    (8,0)*{\cdots};
    (12,0)*{\up{i_m}};
    \endxy
\]
If $s$ is a permutation that can be written as a product of admissible permutations and $\jj=s(\ii)$, then $\ii$ has an $r$-factorization through $\jj$ since all crossings in ${}_{\jj}1_{\ii}$ and ${}_{\ii}1_{\jj}$ are coloured by disconnected vertices, so that $1_{\ii} = {}_{\ii}1_{\jj}{}_{\jj}1_{\ii}$.
\end{example}

\begin{example} \label{example_iij}
The sequence $iij$ has a $3$-factorization through the sequences $\{iji,iji\}$ when $i \cdot j =-1$. The factorization follows from \eqref{eq_iislide2} and \eqref{eq_crossdouble} since
\begin{equation}
 \xy
    (0,0)*{\up{i}};
    (6,0)*{\up{i}};
    (12,0)*{\up{j}};
    \endxy
\;\; = \;\;
-\;
  \xy
    (2,0)*{\twocross{i}{i}};
    (4,2)*{\bullet};(0,-6)*{\bullet};
    (12,0)*{\up{j}};
    \endxy
  \;\;  + \;\;
      \xy
    (2,0)*{\twocross{i}{i}};
    (4,2)*{\bullet};
    (4.3,9)*{\bullet};
    (12,0)*{\up{j}};
    \endxy
\end{equation}
Expanding both terms using \eqref{eq_UUzero} for $i \cdot j =-1$ gives
\begin{equation}\label{eq_examp_iij}
\qquad = \quad
\;\; -\;
 \vcenter{\xy
(-6,-13)*{};
 (6,8)*{};
 (0,0)*{\includegraphics[scale=0.53]{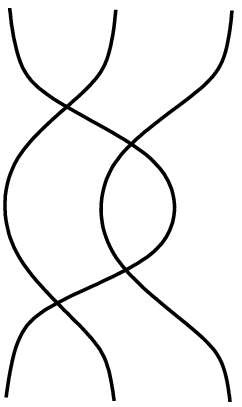}};
 (-5.5,-8)*{\bullet};
 (-7,-12)*{\scs i};(0,-12)*{\scs i};(5.5,-12)*{\scs j};
 \endxy}
\;\; + \;\;  \vcenter{\xy (-6,-13)*{};
 (6,8)*{};
 (0,0)*{\includegraphics[scale=0.53]{j2i.eps}};
 (-0.5,8)*{\bullet};
 (-7,-12)*{\scs i};(0,-12)*{\scs i};(5.5,-12)*{\scs j};
 \endxy}
 \;\; + \;\;
 \xy
    (2,0)*{\twocross{i}{i}};
    (12,2)*{\bullet};(0,-6)*{\bullet};
    (12,0)*{\up{j}};
    \endxy
  \;\;  - \;\;
      \xy
    (2,0)*{\twocross{i}{i}};
    (12,2)*{\bullet};
    (4.3,9)*{\bullet};
    (12,0)*{\up{j}};
    \endxy
\end{equation}
where the last two terms are zero by \eqref{eq_UUzero}.  Explicitly, the factorization is given by
\[
 1_{iij} =  \delta_{1,iij} \left({}_{iij}1_{iji}\right) \left({}_{iji}1_{iij} \right)\delta_{1,iij}x_{1,iij}
 + x_{2,iij}\delta_{1,iij}\left({}_{iij}1_{iji}\right) \left({}_{iji}1_{iij} \right)\delta_{1,iij}.
\]
\end{example}

The following somewhat complex example will be used in the proof of the main theorem.
\begin{example}\label{examp_complex}
Let $\ii = \ii' i_r i_{r+1} i_{r+2} i_{r+3} \dots i_m$ with $i_r=i_{r+2}$, $i_{r+1} \con i_r \con i_{r+3}$, $i_{r+1} \cdot i_{r+a} =0$ for $a\geq 3$, and $i_r \cdot i_{r+b}=0$ for $b\geq 4$.  Observe that
\begin{eqnarray} \label{eq_fff}
 \xy
(0,0)*{\up{i_1}};
 (5,0)*{\cdots};
 (10,0)*{\up{\;\;i_{r}}};
 (18,0)*{\up{\;\;i_{r+1}}};
 (26,0)*{\up{\;\;i_{r}}};
 (31,0)*{\cdots};
 (36,0)*{\up{\;\;i_{m}}};
 \endxy
& \refequal{\eqref{eq_r3_hard}} &
\xy (0,0)*{\up{i_1}};
 (5,0)*{\cdots};(18,0)*{\linecrossL{\;\;i_r}{\;\;i_{r+1}}{\;\;i_r}};  (31,0)*{\cdots};
 (36,0)*{\up{\;\;i_{m}}};\endxy
  \;\; -\;\;
\xy (0,0)*{\up{i_1}};
 (5,0)*{\cdots};(18,0)*{\linecrossR{\;\;i_r}{\;\;i_{r+1}}{\;\;i_r}};  (31,0)*{\cdots};
 (36,0)*{\up{\;\;i_{m}}};\endxy \nn \\
\end{eqnarray}
The first term on the right-hand-side can be rewritten as
\[
\refequal{\eqref{eq_crossdouble}} \quad -\;
\xy (0,0)*{\up{i_1}};
 (5,0)*{\cdots};(18,1.5)*{\includegraphics[angle=180,scale=0.5]{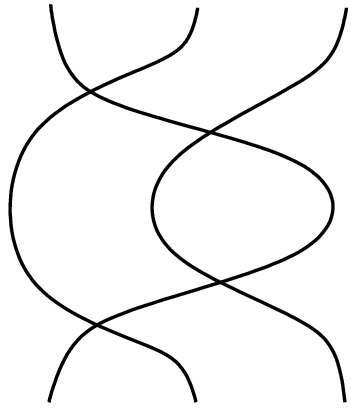}};
 (34,0)*{\up{\;\;i_{r+3}}};
(39,0)*{\cdots};
 (44,0)*{\up{\;\;i_{m}}}; (25.8,1.5)*{\bullet};
  (12,-10.5)*{\scs i_r};(19,-10.5)*{\scs i_{r+1}};(27,-10.5)*{\scs i_r}; \endxy
\quad \refequal{\eqref{eq_UUzero}} \quad -\;
 \xy
 (-20,0)*{\up{\;\;i_{1}}};
  (-15,0)*{\cdots};
 (6,2)*{\includegraphics[scale=0.5]{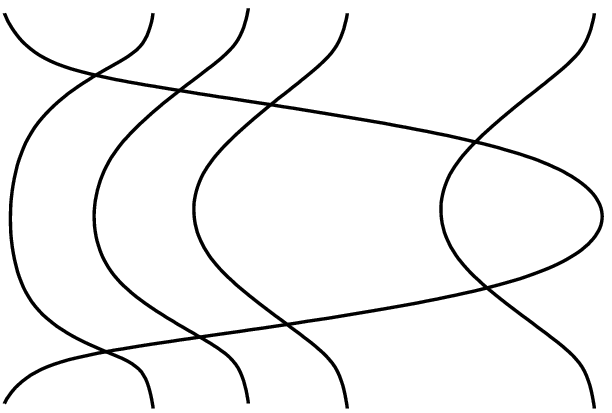}};
(7,0)*{ \cdots};
 (-9,-10.5)*{\scs i_{r}};(-2,-10.5)*{\scs i_{r+1}};(4,-10.5)*{\scs i_{r}};
 (9,-10.5)*{\scs i_{r+3}};(15,-10.5)*{\scs \cdots};
 (22,-10.5)*{\scs i_{m}};
\endxy
\]
and sliding the strand labelled $i_{r+1}$ right, the right-hand-side of \eqref{eq_fff} can be written as
\[
-
 \xy
 (-24,0)*{\up{\;i_{1}}};
  (-20,0)*{\cdots};
 (-2,2)*{\includegraphics[scale=0.45]{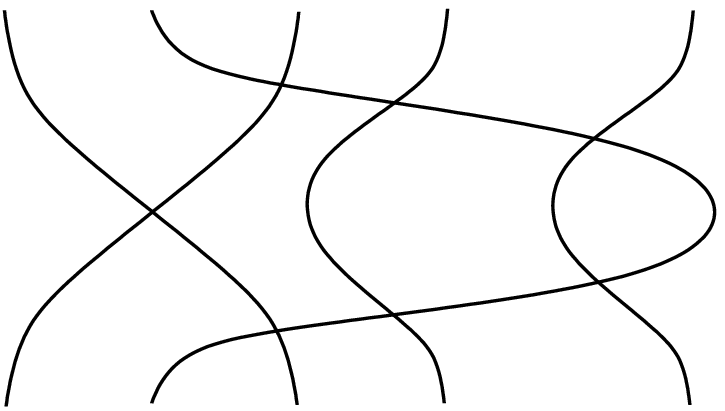}};
(2,2)*{ \cdots};
 (-18,-10.5)*{\scs i_{r}};(-11,-10.5)*{\scs i_{r+1}};(-4.5,-10.5)*{\scs i_{r}};
 (3,-10.5)*{\scs i_{r+3}};(9,-10.5)*{\scs \cdots};
 (13.5,-10.5)*{\scs i_{m}};
\endxy\;
\refequal{\eqref{eq_crossdouble}}
 \xy
 (-24,0)*{\up{\;i_{1}}};
  (-20,0)*{\cdots};
 (0,2)*{\includegraphics[scale=0.45]{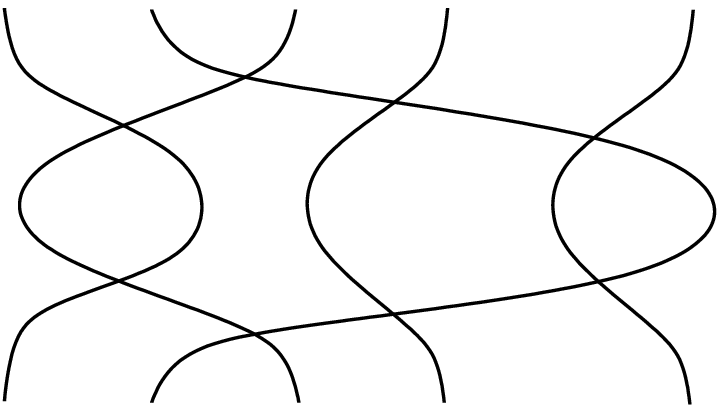}};
(4,2)*{ \cdots};
 (-16,-10.5)*{\scs i_{r}};(-10,-10.5)*{\scs i_{r+1}};(-2,-10.5)*{\scs i_{r}};
 (4.5,-10.5)*{\scs i_{r+3}};(10.5,-10.5)*{\scs \cdots};
 (15.5,-10.5)*{\scs i_{m}}; (-7.4,1.5)*{\bullet};
\endxy\;
\refequal{\eqref{eq_UUzero}}
 \xy
 (-24,0)*{\up{\;i_{1}}};
  (-20,0)*{\cdots};
 (0,2)*{\includegraphics[scale=0.45]{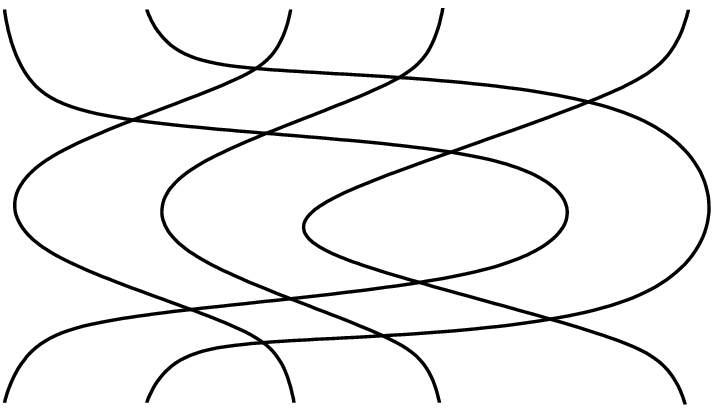}};
(4,1)*{ \cdots};
 (-16,-10.5)*{\scs i_{r}};(-9.5,-10.5)*{\scs i_{r+1}};(-3,-10.5)*{\scs i_{r}};
 (5,-10.5)*{\scs i_{r+3}};(10.5,-10.5)*{\scs \cdots};
 (15.5,-10.5)*{\scs i_{m}};
\endxy
\]
so that $\ii=\ii'i_ri_{r+1}i_r \ii''$ has an $m$-factorization through sequences $\{\jj,\kk\}$ where $\jj= \ii'i_{r+1}i_r \ii''i_r$ and $\kk=\ii'i_r \ii''i_ri_{r+1}$.
\end{example}

Our interest in $r$-factorizations is explained in the following Proposition:
\begin{prop} \label{prop_factorx}
Consider the sequence $\ii$ equipped with an $r$-factorization through sequences $\{\jj_a\}_a$ where $\jj_a=s_a(\ii)$ for permutations $s_a$ in $S_m$. If $x_{s_a(r),\jj}^{\alpha}=0$ for all $a$, then this implies $x_{r,\ii}^{\alpha}=0$.  Furthermore, when $s_a(r)<r$ for all $a$, it is enough to show $x_{s(r),\jj'_a}^{\alpha}=0$ for the truncated sequences $\jj'_a=j_1\dots j_{s_a(r)}$.
\end{prop}

\begin{proof}
We prove the case when $\ii$ has an $r$-factorization through $\jj = s(\ii)$ for some permutation $s$.  The general case is a straight forward extension of this case. Using the $r$-factorization and the fact the $x_{r,\ii}$ commutes with elements in $R_{\ii',i_r, \ii''}$, we can write
\[
x_{r,\ii}^{\alpha} \;\; = \;\; x_{r,\ii}^{\alpha}1_{\ii}   \;\; = \;\; x_{r,\ii}^{\alpha} R_{\ii',i_r, \ii''} \left( {}_{\ii}1_{\jj}\right)\left( {}_{\jj}1_{\ii}\right)R_{\ii',i_r, \ii''}
 \;\; = \;\;  R_{\ii',i_r, \ii''} x_{r,\ii}^{\alpha} \left( {}_{\ii}1_{\jj}\right)\left( {}_{\jj}1_{\ii}\right)R_{\ii',i_r, \ii''}.
\]
Sliding dots through the crossings in ${}_{\ii}1_{\jj}$ using \eqref{eq_ijslide} shows that
\[
x_{r,\ii}^{\alpha} \;\; = \;\;
 R_{\ii',i_r, \ii''} \left( {}_{\ii}1_{\jj}\right)x_{s(r),\jj}^{\alpha}\left( {}_{\jj}1_{\ii}\right)R_{\ii',i_r, \ii''} \;\; = \;\; 0
\]
whenever $x_{s(r),\jj}^{\alpha}=0$.  The second claim in the proposition is clear since $x_{s(r),\jj_a'}^{\alpha}=0$ implies $x_{s(r),\jj_a}^{\alpha}=0$.
\end{proof}

\begin{rem} \label{rem_invariance}
The sequence $\ii=i_1\dots i_m$ factors through the sequence $s_r(\ii)$ for any admissible transposition.  Likewise, $s_r(\ii)$ factors through $\ii$.  Therefore, whenever $r<m-1$ then $x_{m,\ii}^b =0$ if and only if $x_{m,s_r(\ii)}^b=0$.  In particular, $x_{m, \ii}$ and $x_{m,\jj}$ have the same nilpotency degree for any $\jj=\jj'i_m$ with $\conf(\jj)=\conf(\ii)$.
\end{rem}

%
\subsection{Main results}
%

\begin{lem} \label{lem_antichain}
Let $\Lambda= \sum_{i \in I} \lambda_i \cdot i \in \N[I]$. Consider $\ii = i_1 \dots i_m \in \seq(\nu)$ with $m$-stable configuration $a_m(\ii)$ and $m$-antigravity bound $b_m$.  If $\ii$ is an $m$-stable sequence, so that $\conf(\ii)=a_m(\ii)$, then $x_{m,\ii}^{b_m}=0$ in $R_{\nu}^{\Lambda}$.
\end{lem}

\begin{proof}
The proof is by induction on the length $|\ii|=m$.  The base case follows from \eqref{eq_quotient1}.  Assume the result holds for all sequences of the above form with length less than or equal to $m-1$. For the induction step we show that $x_{m,\ii}^{b_m}=0$.  We may assume $\lambda_{i_1} > 0$ otherwise $1_{\ii} =0$ and the result trivially follows. By Remark~\ref{rem_invariance} it suffices to choose a preferred representative for the configuration $\conf(\ii)$.  Choose the representative $\ii = \jj \jj' i_m$ where $\jj=(i_m-r,i_m-(r-1), \dots, i_m-1)$ and $\jj'=(i_m+(m-r-1), i_m+(m-r-2), \dots, i_m+1)$.  It is possible that either $\jj=\emptyset$ or $\jj'=\emptyset$. The idempotent $1_{\ii}$ has the form
\begin{equation}
 1_{\ii} \quad = \quad \xy
 (0,0)*{\up{i_1}};
 (5,0)*{\cdots};
 (10,0)*{\up{\;\;i_{r}}};
 (18,0)*{\up{\;\;i_{r+1}}};
 (23,0)*{\cdots};
 (28,0)*{\up{\;\;i_{m-1}}};
 (36,0)*{\up{\;\;i_{m}}};
 (5,-14)*{\underbrace{\hspace{0.6in}}};
 (23,-14)*{\underbrace{\hspace{0.6in}}};
  (5,-18)*{\jj};
 (23,-18)*{\jj'};
 \endxy \quad \text{and} \qquad \left\{ \;\begin{array}{ccl}
                 i_a \con i_b & \quad & \text{if $b=a+1$ and $a \neq r$} \\
                 & \quad & \text{or $a=r$ and $b=m$,} \\
                 & & \\
                 i_a \cdot i_b = 0 & \quad & \text{otherwise.}
               \end{array} \right. \label{eq_minimalseq}
\end{equation}

First consider the case $\jj' = \emptyset$ so that $|\jj| = r =m-1$.  The definition of $\jj$ is such that  $\conf(\jj) = a_r(\jj)$, so the induction hypothesis implies $x_{r,\jj}^{\delta}=x_{m-1,\jj}^{\delta}=0$, where $\delta=\sum_{j \in {\rm Supp}(\conf(\jj))} \lambda_{j}$.  Since $b_m = \sum_{j \in {\rm Supp}(\conf(\jj i_m))}\lambda_j$ we can write $b_m= \lambda_{i_m} + \delta$ with $\delta >0$ since $\lambda_{i_1} \geq 1$. Using \eqref{eq_UUzero} $x_{m,\ii}^{b_m}$ can be expressed as
\begin{equation}
 x_{m,\ii}^{b_m} \quad = \quad \xy
 (0,0)*{\up{i_1}};
 (5,0)*{\cdots};
 (10,0)*{\up{\;\;i_{m-1}}};
 (18,0)*{\updot{\;\;i_{m}}};
 (22,3)*{b_m};
 \endxy
 \;\; =\;\;
 \xy
 (0,2)*{\reflectbox{\includegraphics[scale=0.5]{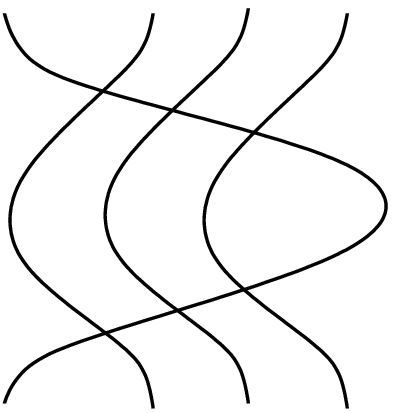}}};
 (10,-10.5)*{\scs i_{m}};(-8,-10.5)*{\scs i_1};(0,-10.5)*{\scs \cdots};
 (8,9.4)*{\bullet}+(7,-1)*{b_m-1};
\endxy
\;\; - \;\;
\xy
 (0,0)*{\up{i_1}};
 (5,0)*{\cdots};
 (10,0)*{\updot{\;\;i_{m-1}}};
 (18,0)*{\updot{\;\;i_{m}}};
 (22,3)*{\qquad b_m-1};
 \endxy
\end{equation}
The first term on the right-hand-side is zero since $b_m-1= \lambda_{i_m}+(\delta-1) \geq \lambda_{i_m}$.  Repeating this argument $\delta$ times on the remaining term above we have
 \begin{equation}
 x_{m,\ii}^{b_m} \quad = \quad
 \xy
 (0,2)*{\reflectbox{\includegraphics[scale=0.5]{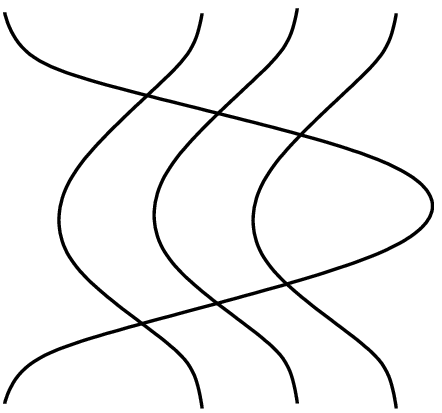}}};
 (12,-10.5)*{\scs i_{m}};(-8,-10.5)*{\scs i_{1}};(0,-10.5)*{\scs \cdots};
 (10,9.9)*{\bullet}+(6,1)*{\scs b_m-\delta};
 (1.5,9.9)*{\bullet}+(4,1)*{\scs \delta-1};
\endxy
\;\;+(-1)^{\delta}\;\;
\xy
 (0,0)*{\up{i_1}};
 (5,0)*{\cdots};
 (10,0)*{\updot{\;\;i_{m-1}}};
 (18,0)*{\updot{\;\;i_{m}}};
 (22,3)*{\qquad b_m-\delta};(13,3)*{\delta};
 \endxy
 \end{equation}
where the first term is zero since $b_m-\delta= \lambda_{i_m}$ and the second term is zero by the induction hypothesis.

It remains to prove the result for $\jj' \neq \emptyset$.  In this case we may assume $\lambda_{i_1} \geq 1$ and $\lambda_{i_{r+1}}\geq 1$ ($\lambda_{i_1} = \lambda_{i_{r+1}}$ if $\jj = \emptyset$), otherwise using that $i_a \cdot i_{r+1}=0$ for all $a \leq r$
\[
 1_{\ii} \quad = \quad \xy
 (0,0)*{\up{i_1}};
 (5,0)*{\cdots};
 (10,0)*{\up{\;\;i_{r}}};
 (18,0)*{\up{\;\;i_{r+1}}};
 (23,0)*{\cdots};
 (28,0)*{\up{\;\;i_{m-1}}};
 (36,0)*{\up{\;\;i_{m}}};
 \endxy \quad = \quad \xy
 (0,2)*{\reflectbox{\includegraphics[scale=0.5]{njpjm.eps}}};
 (10,-10.5)*{\scs i_{r+1}};(-8,-10.5)*{\scs i_1};(-2,-10.5)*{\scs i_2};(2,-10.5)*{\scs \cdots};
 (15,0)*{\cdots};
 (20,0)*{\up{\;\;i_{m-1}}};
 (28,0)*{\up{\;\;i_{m}}};
\endxy
\]
so that $1_{\ii}=0$  in $R_{\nu}^{\Lambda}$ by \eqref{eq_quotient1}, in which case $x_{m,\ii}^{b_m}=0$. Since $\conf(\jj i_m) = a_{r+1}(\jj i_m)$ and $\conf(\jj')=a_{m-r-1}(\jj')$, the induction hypothesis implies that
\[
x_{r+1,\jj i_m}^{\alpha}=x_{m-r-1,\jj'}^{\beta}=0 \qquad \text{for} \quad \alpha= \sum_{j \in{\rm Supp}(\conf(\jj i_m))}\lambda_j, \qquad  \beta=\sum_{j \in{\rm Supp}(\conf(\jj'))}\lambda_j.
\] This implies
\begin{equation} \label{eq_indhyp2}
 \xy
 (-25,0)*{\up{\;\;i_{1}}};
  (-20,0)*{\cdots};
 (-15,0)*{\up{\;\;i_{r}}}; (-16,4)*{\scs };
 (0,2)*{\reflectbox{\includegraphics[scale=0.5]{njpjm.eps}}};
 (10,-10.5)*{\scs i_{m}};(-6,-10.5)*{\scs i_r+1};(0,-10.5)*{\scs \cdots};
 (-9.5,2)*{\bullet}+(-3,0)*{ c};
\endxy \;\; =\;\;0 \qquad \text{for $c \geq \alpha = b_m-\beta$,}
\end{equation}
and using \eqref{eq_UUzero} repeatedly for the disconnected vertices that
\begin{equation} \label{eq_indhyp1}
\xy
 (-16,0)*{\up{i_{1}}};
 (-11,0)*{\cdots};
 (-6,0)*{\up{i_{r}}};
 (0,0)*{\up{i_{r+1}}};
 (5,0)*{\cdots};
 (10,0)*{\updot{\;\;i_{m-1}}};
 (18,0)*{\up{\;\;i_{m}}};
 (22,3)*{\qquad };(13,3)*{b};
 \endxy  = \;\;
 \xy
 (0,2)*{\reflectbox{\includegraphics[scale=0.5]{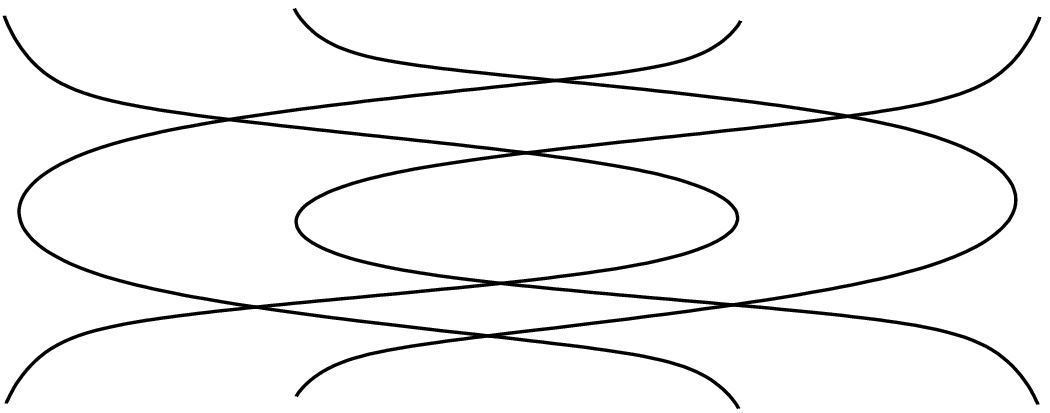}}};
 (-18,2)*{\cdots}; (-18,10)*{\cdots};(-18,-10)*{\cdots};
 (20,-10)*{\cdots}; (19,2)*{\cdots};(19,10)*{\cdots};
 (34,0)*{\up{\;\;i_{m}}};
 (-11,-10.5)*{\scs i_{r}}; (-25,-10.5)*{\scs i_{1}};
 (12.5,-10.5)*{\scs i_{r+1}}; (28,-10.5)*{\scs i_{m-1}};
 (38,3)*{\qquad };(-10.5,1.5)*{\bullet}+(-3,1)*{b};
 \endxy  = \;\;0 \qquad \text{for $b \geq \beta$.}
\end{equation}

The assumption that $\lambda_{i_{r+1}}\geq 1$ implies $\beta \geq 1$. Then
\begin{equation}
 x_{m,\ii}^{b_m} \quad = \quad \xy
 (0,0)*{\up{i_1}};
 (5,0)*{\cdots};
 (10,0)*{\up{\;\;i_{m-1}}};
 (18,0)*{\updot{\;\;i_{m}}};
 (22,3)*{b_m};
 \endxy
 \;\; =\;\;
 \xy
 (-25,0)*{\up{\;\;i_{1}}};
  (-20,0)*{\cdots};
 (-15,0)*{\up{\;\;i_{r}}};
 (0,2)*{\reflectbox{\includegraphics[scale=0.5]{njpjm.eps}}};
 (10,-10.5)*{\scs i_{m}};(-6,-10.5)*{\scs i_r+1};(0,-10.5)*{\scs \cdots};
 (8,9.4)*{\bullet}+(7,-1)*{b_m-1};
\endxy
\;\; - \;\;
\xy
 (0,0)*{\up{i_1}};
 (5,0)*{\cdots};
 (10,0)*{\updot{\;\;i_{m-1}}};
 (18,0)*{\updot{\;\;i_{m}}};
 (22,3)*{\qquad b_m-1};
 \endxy
\end{equation}
where we have used \eqref{eq_UUzero} and the conditions in \eqref{eq_minimalseq}. After sliding the $b_m-1$ dots next to the strand labelled $i_r$ using \eqref{eq_ijslide}, the first term on the right-hand-side is zero by \eqref{eq_indhyp2} since $b_m -1= \alpha + (\beta-1) \geq \alpha$.  Iterating this argument $\beta$ times,
 \[x_{m,\ii}^{b_m} \quad = \quad \xy
 (0,0)*{\up{i_1}};
 (5,0)*{\cdots};
 (10,0)*{\up{\;\;i_{m-1}}};
 (18,0)*{\updot{\;\;i_{m}}};
 (22,3)*{b_m};
 \endxy
 \;\; =
 \xy
 (-25,0)*{\up{\;\;i_{1}}};
  (-20,0)*{\cdots};
 (-15,0)*{\up{\;\;i_{r}}};
 (0,2)*{\reflectbox{\includegraphics[scale=0.5]{njpjmp.eps}}};
 (12,-10.5)*{\scs i_{m}};(-8,-10.5)*{\scs i_{r+1}};(0,-10.5)*{\scs \cdots};
 (10,9.9)*{\bullet}+(6,1)*{\scs b_m-\beta};
 (1.5,9.9)*{\bullet}+(4,1)*{\scs \beta-1};
\endxy
\;\;+(-1)^{\beta}\;\;
\xy
 (0,0)*{\up{i_1}};
 (5,0)*{\cdots};
 (10,0)*{\updot{\;\;i_{m-1}}};
 (18,0)*{\updot{\;\;i_{m}}};
 (22,3)*{\qquad b_m-\beta};(13,3)*{\beta};
 \endxy \]
where the first diagram is zero by \eqref{eq_indhyp2} and the second term is zero by \eqref{eq_indhyp1}.
\end{proof}

\begin{thm} \label{thm_main}
Let $\Lambda= \sum_{i \in I} \lambda_i \cdot i \in \N[I]$  and $\ii=i_1 \dots
i_r\dots  i_m \in \seq(\nu)$.  Then the nilpotency degree of
$x_{r,\ii}$ in $R_{\nu}^{\Lambda}$ is less than or equal to the $r$-antigravity bound $$b_r = \sum_{j \in {\rm Supp}(a_r(\ii))} \lambda_j.$$
\end{thm}

\begin{proof}
The proof is by induction on the length $|\ii|=m$.  The base case follows from \eqref{eq_quotient1}.  Assume the result holds for all sequences of the above form with length less than or equal to $m-1$.    We show that $x_{m,\ii}^{b_m} = 0$. For the induction step we show that one of the following must be true.
\begin{enumerate}
  \item  The sequence $\ii$ is $m$-stable, that is, $\conf(\ii)=a_m(\ii)$.
  \item  The nilpotency of the sequence $\ii$ is bound above by the nilpotency of a sequence $s(\ii)=\jj=j_1\dots j_m$ for some permutation $s\in S_m$ with $s(m)<m$.  Furthermore, the $s(m)$-antigravity bound for $\jj$ is the same as the $m$-antigravity bound $b_m$ for $\ii$.
\end{enumerate}
In the first case the theorem follows by Lemma~\ref{lem_antichain}, and in the second case the theorem follows from the induction hypothesis applied to the truncated sequence $\jj'=j_1 \dots j_{s(m)}$.

Consider the sequence $\ii=i_1\dots i_m$ in antigravity with anchored bead $i_m$.  If any beads are removed from the $\Gamma$-abacus in $m$-antigravity let the $i_a$ be the first bead to be removed.  This means that $i_a$ cannot be connected in $\Gamma$ to any $i_{a'}$ for $a' > a$, so that the sequence $\ii$ has an $m$-factorization through the sequence $s(\ii)=\jj:=i_1 \dots i_{a-1} i_{a+1} \dots i_m i_a$. It is clear that the $s(m)-$antigravity bound $b$ for the truncated sequence $\jj'=i_1\dots i_{a-1} i_{a+1} \dots i_m$ is the same as the antigravity bound $b_m$ for $\ii$.  The induction hypothesis implies that $x_{s(m),\jj'}^{b}=x_{s(m),\jj'}^{b_m}=0$, so $x_{m,\ii}^{b_m}=0$ by Proposition~\ref{prop_factorx} since $\ii$ has an $m$-factorization through $\jj$.   Thus, it suffices to assume that all beads are trapped below the anchored bead $i_m$ in antigravity.

Consider the rightmost $r$ such that $i_r \dots i_{m-1}i_m$ is not $m$-stable.  The antigravity configuration must contain one of the following unstable forms:
\begin{equation} \label{proof_anti}
 \xy
  (0,0)*{\includegraphics[scale=0.8]{anti1.eps}};
  (0,-4)*{r};
 \endxy \qquad \quad \xy
  (0,0)*{\includegraphics[scale=0.8]{anti2.eps}}; (-2,-4)*{r};
 \endxy  \qquad  \quad  \xy
  (0,0)*{\includegraphics[angle=180, scale=0.8]{anti2.eps}}; (2,-4)*{r};
 \endxy
 \qquad \quad
 \xy
  (0,0)*{\includegraphics[scale=0.8]{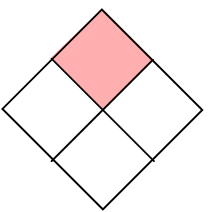}}; (0,-4)*{r};(0,4)*{\textbf{\textit{m}}};
 \endxy
\end{equation}
If the unstable configuration has the first form in \eqref{proof_anti}, then if the top box is the anchor Remark~\ref{rem_invariance} implies it suffices to assume $r=m-1$ for the first configuration. Proposition~\ref{prop_anchorc} then implies that the nilpotency degree of the sequence $x_{m,\ii}$ is bound above by the nilpotency degree of $x_{m-1,\jj'}$ for the shorter sequence $\jj'$ given by truncating $\ii$ at the $(m-1)$st term.  Because sequences $\ii$ and $\jj'$ are related by a stack antigravity move their antigravity bounds are the same.  Hence, the result follows by the induction hypothesis.

If the unstable configuration has the first form in \eqref{proof_anti} and the top box is not the anchor, then by Remark~\ref{rem_invariance} it suffices to consider the representative of $\conf(\ii)$ where the upper box corresponds to the $(r+1)$st bead, that is, $i_r=i_{r+1}$ so that $\ii$ has the form $\ii = \ii'i_r i_{r} i_{r+2} \ii''$ with $i_r \con i_{r+2}$.  As in Example~\ref{example_iij}, we can write
\begin{eqnarray} \nn
 \xy
(0,0)*{\up{i_1}};
 (5,0)*{\cdots};
 (10,0)*{\up{\;\;i_{r}}};
 (18,0)*{\up{\;\;i_{r}}};
 (26,0)*{\up{\;\;i_{r+2}}};
 (31,0)*{\cdots};
 (36,0)*{\up{\;\;i_{m}}};
 \endxy
\refequal{\eqref{eq_examp_iij}} \quad -\;
\xy (0,0)*{\up{i_1}};
 (5,0)*{\cdots};(18,1.5)*{\includegraphics[angle=180,scale=0.5]{ccc2.eps}};
(31,0)*{\cdots};
 (36,0)*{\up{\;\;i_{m}}}; (12,-7)*{\bullet};
  (12,-10.5)*{\scs i_r};(19,-10.5)*{\scs i_{r}};(27,-10.5)*{\scs i_{r+2}}; \endxy
\;\; + \;\; \xy (0,0)*{\up{i_1}};
 (5,0)*{\cdots};(18,1.5)*{\includegraphics[angle=180,scale=0.5]{ccc2.eps}};
(31,0)*{\cdots};
 (36,0)*{\up{\;\;i_{m}}}; (17.5,9)*{\bullet};
  (12,-10.5)*{\scs i_r};(19,-10.5)*{\scs i_{r}};(27,-10.5)*{\scs i_{r+2}}; \endxy
\end{eqnarray}
Since $i_r$ is the first vertex where one of the configurations in \eqref{proof_anti} appears,  we can assume that $i_r$ is not connected to any of the vertices in $\ii''$.  Pulling the strand labelled $i_r$ to the far right gives an $m$-factorization
\[
 \xy
(0,0)*{\up{i_1}};
 (5,0)*{\cdots};
 (10,0)*{\up{\;\;i_{r}}};
 (18,0)*{\up{\;\;i_{r}}};
 (26,0)*{\up{\;\;i_{r+2}}};
 (31,0)*{\cdots};
 (36,0)*{\up{\;\;i_{m}}};
 \endxy \;\; = \;\;
 -\;
 \xy
 (-20,0)*{\up{\;\;i_{1}}};
  (-15,0)*{\cdots};
 (6,2)*{\includegraphics[scale=0.5]{ccc4.eps}};
(7,0)*{ \cdots};
 (-9,-10.5)*{\scs i_{r}};(-2,-10.5)*{\scs i_{r}};(3.5,-10.5)*{\scs i_{r+2}};
 (9.5,-10.5)*{\scs i_{r+3}};(15,-10.5)*{\scs \cdots};
 (22,-10.5)*{\scs i_{m}}; (-8,-6.5)*{\bullet};
\endxy
\;\; + \;\;
 \xy
 (-20,0)*{\up{\;\;i_{1}}};
  (-15,0)*{\cdots};
 (6,2)*{\includegraphics[scale=0.5]{ccc4.eps}};
(7,0)*{ \cdots};
 (-9,-10.5)*{\scs i_{r}};(-2,-10.5)*{\scs i_{r}};(3.5,-10.5)*{\scs i_{r+2}};
 (9.5,-10.5)*{\scs i_{r+3}};(15,-10.5)*{\scs \cdots};
 (22,-10.5)*{\scs i_{m}}; (-2.5,10)*{\bullet};
\endxy
\]
of $\ii$ through copies of the sequence  $s(\ii):=\;\ii'i_r i_{r+2} \ii'' i_{r}$ where $s(m)=m-1$.  Applying the induction hypothesis to the truncated sequence $\jj':=\ii'i_ri_{r+2}\ii''$ implies $x_{s(m),\jj'}^b =0$, where $b$ is the $s(m)$-antigravity bound for the sequence $\jj'$.  Since $x_{s(m),\jj'}^b =0$ implies that $x_{s(m),\jj'i_r}^b =0$, the factorization of $\ii$ through $\jj'i_m$ implies $x_{m,\ii}^b=0$. However, the sequence $\jj'$ is obtained from $\ii$ by applying a stack antigravity move.  Therefore, the $s(m)$-antigravity bound $b$ for the sequence $\jj'$ is the same as the $m$-antigravity bound $b_m$ for the sequence $\ii$ and $x_{m,\ii}^{b_m}=0$ as desired.

If the unstable configuration has the second or third form of \eqref{proof_anti}, then by Remark~\ref{rem_invariance} it suffices to consider the representative of $\conf(\ii)$ of one of the two forms
\[
\xy
  (0,0)*{\includegraphics[scale=.9]{anti2.eps}};
  (-2,-5)*{ r}; (2.5,0)*{\scs r+1};(-2.5,4.5)*{\scs r+2};
 \endxy  \qquad \text{or} \qquad  \xy
  (0,0)*{\includegraphics[angle=180, scale=0.9]{anti2.eps}};
(2,-5)*{ r}; (-2.5,0)*{\scs r+1};(2.5,4.5)*{\scs r+2};
 \endxy
\]
In either case, Example \eqref{examp_complex} gives an $m$-factorization of
$\ii=\ii'i_ri_{r+1}i_r \ii''$ through sequences $\{\jj,\kk\}$ where $\jj= s(\ii)= \ii'i_{r+1}i_r \ii''i_r$ and $\kk=s'(\ii)=\ii'i_r \ii''i_ri_{r+1}$.  By examining  the resulting configurations it is easy to see that the $s(m)$-antigravity bound $b$ for $\jj$ is greater than or equal to the $s'(m)$-antigravity bound $b'$ for $\kk$. Hence, if we set $\jj'=\ii'i_{r+1}i_r \ii''$ and $\kk' = \ii'i_r \ii''$ then by the induction hypothesis both $x_{s(m),\jj'}^b=x_{s'(m),\kk'}^b=0$, implying $x_{s(m),\jj}^b=x_{s'(m),\kk}^b=0$.  But the sequence $\jj$ is obtained from the sequence $\ii$ by applying an antigravity $L$-move.  Therefore, $b=b_m$ and $x_{m,\ii}^{b_m}=0$ by Proposition~\ref{prop_factorx}.

If the unstable configuration has the last form in \eqref{proof_anti}, then by Remark~\ref{rem_invariance} it suffices to consider the representative of $\conf(\ii)$ of the form
\[
\xy
  (-2.5,0)*{\includegraphics[scale=.9]{square-move.eps}};
  (-2.5,-5)*{\scs  m-3}; (-7.5,0)*{\scs m-1};(2.5,0)*{\scs m-2};(-2.5,4.5)*{\scs m};
 \endxy
\]
so that $\ii = \ii'i_m i_{m-2}i_{m-1}i_m$ for some sequence $\ii'$.
Working locally around these last four strands, repeatedly apply \eqref{eq_UUzero} to slide all the dots from right to left, so that $x_{m,\ii}^{b_m}$ can be rewritten as
\begin{eqnarray}
  \sum_{\ell_1 + \ell_2 = b_m-1}
\xy
 (-6,0)*{\up{i_m}};
 (2,0)*{\up{\;\;i_{m-2}}};
(13,0)*{\twocross{\;\;i_{m-1}}{\;\;i_m}};
(-6,0)*{\bullet}+(3,1)*{\scs \ell_1};
(12,0)*{\bullet}+(-3,1)*{\scs \ell_2};
 \endxy
 \;\;-\;\;
 \sum_{\ell_1 + \ell_2 = b_m-1}
 \xy
 (0,0)*{\includegraphics[scale=0.5]{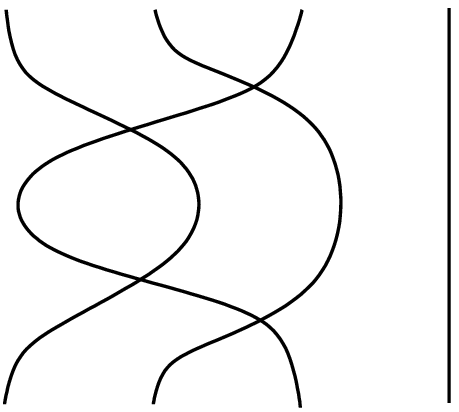}};
 (-11,-12.5)*{\scs i_{m}};(-4,-12.5)*{\scs i_{m-1}};(5,-12.5)*{\scs i_{m-2}};(12,-12.5)*{\scs i_{m}};
 (11,0)*{\bullet}+(3,1)*{\scs \ell_2};
 (-1.5,0)*{\bullet}+(3,1)*{\scs \ell_1};
\endxy
 \;\;+\;\;
 \xy
 (-6,0)*{\updot{i_m}};
 (2,0)*{\up{\;\;i_{m-2}}};
 (10,0)*{\up{\quad i_{m-1}}};
 (18,0)*{\up{\;\;i_{m}}};
 (-9.4,3)*{b_m};
 \endxy \nn \\  \label{eq_dproof}
\end{eqnarray}
Using \eqref{eq_r3_hard}, the first two terms above have $(m-3)$-factorizations through the sequences
\[ \jj_1=\ii'i_{m-2}i_{m-1}i_mi_m, \qquad \jj_2=\ii'i_{m-2}i_mi_mi_{m-1}, \qquad
\jj_3=\ii'i_mi_mi_{m-2}i_{m-1}.\]

By Proposition~\ref{prop_subset} and the induction hypothesis, we have that $x_{m-a,\jj_a}^{b_m}=0$ for $a \in \{1,2,3\}$ and that $x_{m-3,\ii'i_{m}}^{b_m}=0$.  The first two terms in \eqref{eq_dproof} are zero because they can be written as a linear combination of terms that contain the local configuration
\begin{equation}
 \sum_{ \xy (0,0)*{\scs \ell_1 +\ell_{2}}; (0,-3)*{\scs =b_m-1}; \endxy}
 \; \xy (-3.5,0)*{\supdot{\; i_m}};
    (3.5,0)*{\supdot{\; i_m}}; (-1,2)*{\scs \ell_1};(6,2)*{\scs \ell_2};
    \endxy
\end{equation}
and are therefore equal to zero by Proposition~\ref{prop_symcomb}.  The third term in \eqref{eq_dproof} is zero since we have shown $x_{m-3,\ii'i_{m}}^{b_m}=0$. Hence, all terms in \eqref{eq_dproof} are zero showing that the $x_{m,\ii}^{b_m}=0$.

Finally, if none of the unstable configurations in \eqref{proof_anti} occur, then the configuration $\conf(\ii)$ is the same as $a_m(\ii)$, so the result holds by Lemma~\ref{lem_antichain}.
\end{proof}

It is clear that the antigravity bound $b_r$ for the nilpotency of $x_{r,\ii}$ in $R_{\nu}^{\Lambda}$ is always less than or equal to the level $\ell(\Lambda)$.  Therefore, we have the following Corollary to Theorem~\ref{thm_main}.

\begin{cor}[Brundan--Kleshchev Conjecture]
If $\ell=\ell(\Lambda)$ is the level of $\Lambda$,  then $x_{r,\ii}^{\ell}=0$ in $R_{\nu}^{\Lambda}$ for any sequence $\ii\in \seq(\nu)$ and any $1 \leq r \leq m$.
\end{cor}

\begin{rem}
In general the antigravity bound is not tight.  For example, Proposition~\ref{prop_symcomb} shows that if $\conf(\ii)$ contains a sub-configuration of the form
\[
 \xy
 (0,0)*{\includegraphics[scale=0.5]{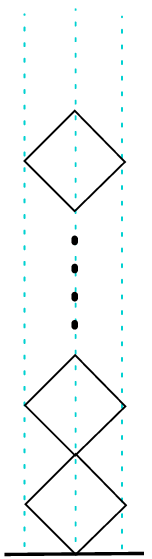}};
 (0,-16)*{\scs i}; (5,-3)*{\left.\begin{array}{c}\; \\\;  \\ \; \\ \; \\\; \end{array} \right\}}; (15,-3)*{\lambda_{i}+1};
 \endxy
\]
then the idempotent $1_{\ii}=0$ in $R_{\nu}^{\Lambda}$, so that $x_{r,\ii}=0$ for all $r$.  More generally, if for any term $i_r$ the configuration $\conf(\ii)$ has a local configuration of the form
\[
 \xy
 (0,0)*{\includegraphics[scale=0.5]{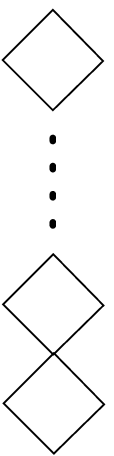}};
 (0,-9)*{r}; (5,0)*{\left.\begin{array}{c}\; \\\;  \\ \; \\ \; \\\; \end{array} \right\}}; (15,1)*{b_r+1};
 \endxy
\]
then Proposition~\ref{prop_symcomb}, together with Theorem~\ref{thm_main}, imply $1_{\ii} =0$ in $R_{\nu}^{\Lambda}$.  Furthermore,  if after applying antigravity moves to $\conf(\ii)$ a configuration of the above form appears, then it is not hard to check that $1_{\ii}=0$ in $R_{\nu}^{\Lambda}$.

We do not know of any sequences $\ii$ where $1_{\ii} \neq 0$ and the antigravity bound is not tight.
\end{rem}


\bibliographystyle{alpha}

%

\vspace{0.1in}

\noindent A.H.:  { \sl \small Department of Mathematics, University of California, Riverside, Riverside, CA 92521} \newline \noindent
  {\tt \small email: alex@math.ucr.edu}

\vspace{0.1in}

\noindent A.L.:  { \sl \small Department of Mathematics, Columbia University, New
York, NY 10027} \newline \noindent
  {\tt \small email: lauda@math.columbia.edu}

%
\end{document}